\documentclass[12pt]{amsart}
\usepackage{amsmath, amssymb, amsfonts}
\usepackage[all]{xypic}
\usepackage[pdftex]{graphicx}
\usepackage{xcolor}
\definecolor{newgrey}{rgb}{0.95,0.95,0.95}
\usepackage[color=newgrey]{todonotes}
\usepackage{leftindex}

\theoremstyle{plain}
\newtheorem{theorem}{Theorem}[section]
\newtheorem{lemma}[theorem]{Lemma}
\newtheorem{corollary}[theorem]{Corollary}
\newtheorem{prop}[theorem]{Proposition}

\theoremstyle{remark}

\newtheorem{remark}[theorem]{Remark}

\newtheorem*{note*}{Note}
\newtheorem*{remark*}{Remark}
\newtheorem*{example*}{Example}

\theoremstyle{definition}
\newtheorem*{definition*}{Definition}
\newtheorem*{hypothesis*}{Hypothesis}
\newtheorem*{assumptions*}{Assumptions}
\newtheorem{definition}[theorem]{Definition}

\newcommand{\Z}{\mathbb{Z}}
\newcommand{\R}{\mathbb{R}}
\newcommand{\Q}{\mathbb{Q}}
\newcommand{\C}{\mathbb{C}}

\newcommand{\Aut}{\mathrm{Aut}}
\newcommand{\Gal}{\mathrm{Gal}}
\newcommand{\Tr}{\mathrm{Tr}}

\newcommand{\Norm}{\mathrm{Norm}}
\newcommand{\rank}{\mathrm{rank}}

\newcommand{\Hom}{\mathrm{Hom}}

\newcommand{\tr}{\mathrm{tr}}

\newcommand{\tors}{\mathrm{tors}}
\newcommand{\Irr}{\mathrm{Irr}}

\numberwithin{equation}{section}

\usepackage{bbm}
\usepackage{tikz-cd}

\newcommand{\MM}{\mathcal{M}}
\newcommand{\OO}{\mathcal{O}}

\newcommand{\Cl}{\operatorname{Cl}}

\newcommand{\z}{\mathbb{Z}}
\newcommand{\q}{\mathbb{Q}}

\newcommand{\ot}{\otimes}
\newcommand{\rk}{\mathrm{rk}}

\newcommand{\one}{\mathbbm{1}}

\newcommand{\opp}{\mathrm{op}}

\newcommand{\lMM}[1]{\leftindex^{\MM}{#1}}
\newcommand{\lGamma}[1]{\leftindex^{\Gamma}{#1}}

\title[{On the existence of free sublattices of bounded index}]{On the existence of free sublattices of\\
bounded index and arithmetic applications}

\author{Henri Johnston}
\address{
Department of Mathematics and Statistics\\
University of Exeter\\
Exeter\\
EX4 4QF\\
United Kingdom
}
\email{H.Johnston@exeter.ac.uk}
\urladdr{https://mathematics.exeter.ac.uk/staff/hj241}

\author{Alex Torzewski}
\address{
Department of Mathematics\\
King's College London\\
Strand\\
London\\ 
WC2R 2LS\\
United Kingdom\\
and Heilbronn Institute for Mathematical Research\\
Bristol\\
BS8 1UG\\
United Kingdom 
}
\email{alex.torzewski@gmail.com}
\urladdr{https://nms.kcl.ac.uk/alex.torzewski/}

\date{Version of 24th May 2024}
\begin{document}

\maketitle

\begin{abstract}
Let $\mathcal{O}$ be a Dedekind domain whose field of fractions $K$ is a global field.
Let $A$ be a finite-dimensional separable $K$-algebra and let $\Lambda$ be an $\mathcal{O}$-order 
in $A$. Suppose that $X$ is a $\Lambda$-lattice such that 
$K \otimes_{\mathcal{O}} X$ is free of some finite rank $n$ over $A$.
Then $X$ contains a (non-unique) free $\Lambda$-sublattice of rank $n$.
The main result of the present article is to show there exists such a sublattice $Y$
such that the generalised module index $[X : Y]_{\mathcal{O}}$ has explicit upper bounds with
respect to division
that are independent of $X$ and can be chosen to satisfy certain conditions.
We give examples of applications to the approximation of normal integral bases 
and strong Minkowski units, and to the Galois module structure of rational points over abelian varieties.
\end{abstract}

\section{Introduction}

Let $A$ be a finite-dimensional semisimple $\Q$-algebra and let $\Lambda$ be an order in $A$. For example, if $G$ is a finite group, then the group ring $\Z[G]$ is an order in the group algebra 
$\Q[G]$.
A $\Lambda$-lattice is a (left) $\Lambda$-module that is finitely generated and torsion-free over $\Z$.
A special case of the Jordan--Zassenhaus theorem says that for each positive integer $t$, there are only
finitely many isomorphism classes of $\Lambda$-lattices of $\Z$-rank at most $t$.

Now fix a positive integer $n$. 
Then there exists a positive integer $m$ with the following property:
given any $\Lambda$-lattice $X$ such that $\Q \otimes_{\Z} X$ is free of rank $n$ as an $A$-module, there exists a free $\Lambda$-sublattice $Y$ of $X$ such that the index $[X:Y]$ is at most $m$. 
To see this, first note that by clearing denominators of a free basis of $\Q \otimes_{\Z} X$ over $A$, 
any such $X$ must contain a (non-unique) free $\Lambda$-sublattice of rank $n$, 
necessarily of finite index $m_{X}$ in $X$.
Since the Jordan--Zassenhaus theorem implies that there are only finitely many choices for $X$ up to isomorphism, we may take $m$ to be the maximal $m_{X}$ as $X$ ranges over all such choices.
Masser--W\"ustholz \cite{MR1361754,MR1336608} defined the \emph{class index} $i_{n}(\Lambda)$
to be the smallest possible value of $m$. 
Using methods from the geometry of numbers, they were able to provide upper bounds for $i_{n}(\Lambda)$
in special cases that led to results on the existence of isogenies between abelian varieties of small degrees (see also \cite{MR1619802}).

We can in fact consider bounds that are also upper bounds with respect to division.
In the above argument, we can instead take $m$ to be any common multiple of the $m_{X}$.
Then $m$ has the following
property: given any $\Lambda$-lattice $X$ such that $\Q \otimes_{\Z} X$ is free of rank $n$ as an $A$-module, there exists a free $\Lambda$-sublattice $Y$ of $X$ such that $[X:Y]$ divides $m$. 
The main goals of the present article are to give explicit choices of $m$ with this property
and to give examples of arithmetic applications. 
In fact, the setting generalises to the case in which $\Lambda$ is an $\mathcal{O}$-order
where $\mathcal{O}$ is a Dedekind domain whose field of fractions $K$ is a global field assumed
not to be equal to $\mathcal{O}$,
and $A$ is a finite-dimensional semisimple $K$-algebra. 
In this situation, the group index $[X:Y]$ is replaced by the generalised module index 
$[X:Y]_{\mathcal{O}}$. 
The main result, Theorem~\ref{thm:main}, gives upper bounds for this index with respect to division
that are independent of $X$ and can be chosen to satisfy certain conditions. 
The proof of this result requires the hypothesis that $A$ is a separable $K$-algebra;
if $K$ is of characteristic zero, then this follows automatically from the assumption that $A$ is semisimple.

We now give examples of the algebraic results and arithmetic applications.
The following result is a weaker version of Theorem~\ref{thm:gpringmain} 
obtained via specialisation and Remark~\ref{rmk:crude-ZG-max-index-bound}.

\begin{theorem}\label{thm:intro-group-ring}
Let $G$ be a finite group and let $k$ be a positive integer.
Then there exists a positive integer $i$, which can be chosen to be coprime to $k$, 
with the following property:
given any $\Z[G]$-lattice $X$ such that $\Q \otimes_{\Z }X$ is free of rank $n$ over $\Q[G]$,
there exists a free $\Z[G]$-sublattice 
$Z$ of $X$ such that the index $[X : Z]$ divides $i \cdot |G|^{\lceil 3|G|/2 \rceil n}$.
\end{theorem}

In Theorem \ref{thm:gpringmain}, we also give conditions on $G$ under which we can 
take $i=1$ (see also \S \ref{subsec:integral-group-rings}).
The term $|G|^{\lceil 3|G|/2 \rceil n}$ is a crude but neat upper bound 
for a more precise expression that will be made explicit.
The following result is Theorem \ref{thm:rat-rep-groups}, which is just one example of the stronger results that can be obtained in special cases. 

\begin{theorem}\label{thm:intro-rat-rep-groups}
Let $G$ be a finite group and suppose that there exist positive integers $t, n_{1}, \ldots, n_{t}$ 
such that  $\Q[G] \cong \prod_{i=1}^{t} \mathrm{Mat}_{n_{i}}(\Q)$. 
If $X$ is an $\Z[G]$-lattice such that $\Q \otimes_{\Z} X$ is free of rank $n$ over $\Q[G]$,
then there exists a free $\Z[G]$-sublattice $Z$ of $X$ such that $[X : Z]$ divides 
\[
\left( 
|G|^{|G|}
{\prod_{i=1}^{t}} n_{i}^{-n_{i}^{2}}
\right)^{\frac{3n}{2}}.
\]
\end{theorem}

Before sketching the ideas used in the proof of the main result, we discuss how a variant of Theorem \ref{thm:intro-group-ring} can be applied in the following arithmetic situation.
Let $L/K$ be a finite Galois extension such that $K$ is equal to either $\Q$ or an imaginary 
quadratic field. Let $G=\Gal(L/K)$ and let $\mu_{L}$ denote the roots of unity of $L$.
In this setting, $\mathcal{O}_{L}^{\times}/\mu_{L}$ is a $\Z[G]$-lattice and 
one can show that $L/K$ has a so-called \textit{Minkowski unit},
that is, an element $\varepsilon \in \mathcal{O}_{L}^{\times}/\mu_{L}$ such that
$\Q \otimes_{\Z} (\mathcal{O}_{L}^{\times}/\mu_{L}) = \Q[G] \cdot \varepsilon$.
Such an $\varepsilon$ is said to be a \textit{strong Minkowski unit} if 
$\mathcal{O}_{L}^{\times}/\mu_{L} = \Z[G] \cdot \varepsilon$.
The existence of strong Minkowski units (which some authors refer to as 
Minkowski units) has been studied in numerous articles; see Remark \ref{rmk:exist-strong-Minkowski}.
In \S \ref{sec:application-approximation-of-strong-minkowski-units}, we give several results on 
the approximation of strong Minkowski units.
The following result is a weakening of Theorem \ref{thm:main-Minkowski} obtained via Remark \ref{rmk:weak-max-order-bound-Minkowski}.

\begin{theorem}\label{thm:intro-Minkowski}
Let $G$ be a finite group and let $k$ be a positive integer. 
Then there exists a positive integer $i$, which can be chosen to be coprime to $k$, with the following property:
given any finite Galois extension $L/K$ with $\Gal(L/K)\cong G$ and $K$ equal to either $\Q$ or an imaginary quadratic field,
there exists a Minkowski unit $\varepsilon \in \mathcal{O}_{L}^{\times}/\mu_{L}$ 
such that the index $[\mathcal{O}_{L}^{\times}/\mu_{L} : \Z[\Gal(L/K)] \cdot \varepsilon]$ 
divides $i \cdot |G|^{\lceil 3|G|/2 \rceil -2}$.
\end{theorem}

Again, stronger results can be obtained in special cases.
Analogous applications to the approximation of normal integral bases are given in
\S \ref{sec:approx-NIB} and to the Galois module structure of rational points on abelian varieties
are given in \S \ref{sec:application-rational-points-on-abelian-varieties}.

We now outline the ideas used in the proof of the main result Theorem \ref{thm:main}.
Let $\mathcal{O}$ be a Dedekind domain whose field of fractions $K$ is a global field and assume that 
$\mathcal{O} \neq K$. Let $A$ be a finite-dimensional separable $K$-algebra and let $\Lambda$
be an $\mathcal{O}$-order in $A$. 
Let $\mathcal{M}$ be a maximal $\mathcal{O}$-order in $A$ containing $\Lambda$. 
Note that the existence of $\mathcal{M}$ is ensured by the 
separability hypothesis on $A$ and that the choice of $\mathcal{M}$ need not be unique.
Let $X$ be a $\Lambda$-lattice such that $K \otimes_{\mathcal{O}} X$ is free of rank $1$ over $A$
(the higher rank case is similar).
We consider the unique $\mathcal{M}$-lattice ${\leftindex^{\mathcal{M}}{X}}$ contained in $X$
that is maximal with respect to inclusion.
Then ${\leftindex^{\mathcal{M}}{X}}$ is locally free over $\mathcal{M}$, and as explained in 
Corollary \ref{cor:free-sublattice-of-lf-max-case}, ${\leftindex^{\mathcal{M}}{X}}$ contains
a free $\mathcal{M}$-sublattice $\mathcal{M} \cdot \varepsilon$ 
with an index that can be controlled (the key ingredients here
are the Jordan--Zassenhaus theorem and Roiter's lemma).
Hypotheses on $\mathcal{M}$ can also be given to ensure that this index is trivial
(see Lemma \ref{lem:equiv-locally-free-implies-free}). 
We then obtain a bound on the index $[X : \Lambda \cdot \varepsilon]_{\mathcal{O}}$
by taking the product of bounds on the indices corresponding to each of the three inclusions
\[  
\Lambda \cdot \varepsilon \subseteq \mathcal{M} \cdot \varepsilon
\subseteq \leftindex^{\MM}{X} \subseteq X.
\]
Note that $[\mathcal{M} \cdot \varepsilon : \Lambda \cdot \varepsilon]_{\mathcal{O}} = [\mathcal{M} : \Lambda]_{\mathcal{O}}$, which is equal to the product of the indices of the localisations of $\mathcal{M}$
and $\Lambda$.
Moreover, $[X : \leftindex^{\MM}{X}]_{\mathcal{O}}$ divides $[\mathcal{M}X : \leftindex^{\MM}{X}]_{\mathcal{O}}$, where $\mathcal{M}X$ is the unique $\mathcal{M}$-lattice containing $X$
that is minimal with respect to inclusion. In Corollary~\ref{cor:bddbyconductor}, we show that 
$[\mathcal{M}X : \leftindex^{\MM}{X}]_{\mathcal{O}}$ divides $[\mathcal{M} : J]_{\mathcal{O}}$
where $J$ is any full two-sided ideal of $\mathcal{M}$ contained in $\Lambda$.
Again, $[\mathcal{M} : J]_{\mathcal{O}}$ can be computed by localisation. 
Crucially, the product of bounds of indices obtained is independent of the choice of $\Lambda$-lattice $X$. 

If $G$ is a finite group such that $|G|$ is invertible in $K$ and 
$\Lambda=\mathcal{O}[G]$, then $J$ can be taken to be the (left) conductor of
$\mathcal{M}$ into $\Lambda$ (the left and right conductors are equal in this case)
and $[\mathcal{M} : J]_{\mathcal{O}}$ 
can be computed explicitly using Jacobinski's conductor formula \cite{MR204538}.
We also obtain an explicit formula for  $[\mathcal{M} : \mathcal{O}[G]]_{\mathcal{O}}$,
which may be of independent interest. 
Note that in the setting of Theorem \ref{thm:intro-group-ring} with $n=1$, 
the term $|G|^{\lceil 3|G|/2 \rceil}$
is a crude but neat upper bound for $[\mathcal{M}: \Z[G]] \cdot [\mathcal{M}: J] = [\mathcal{M}:\Z[G]]^{3}$
and the term $i$ is the upper bound for $[\leftindex^{\mathcal{M}}{X} : \mathcal{M} \cdot \varepsilon]$
given by Corollary \ref{cor:free-sublattice-of-lf-max-case}.
Moreover, we can take $i=1$ when $\mathcal{M}$ satisfies the equivalent conditions of Lemma \ref{lem:equiv-locally-free-implies-free} (see \S \ref{subsec:integral-group-rings} for conditions on $G$ under which this holds).

\subsection*{Acknowledgements}
The authors are grateful to Werner Bley, Nigel Byott, Frank Calegari, Hebert Gangl, 
Tommy Hofmann, Donghyeok Lim, Daniel Macias Castillo, Alexandre Maksoud 
and John Nicholson for helpful comments and discussions. 
The authors are indebted to the anonymous referee for carefully reading the manuscript, for corrections to the statements of Proposition \ref{prop:trivialclassgroup} and Lemma 
\ref{lem:Hebrand-special-case}, and for drawing their attention to the work of Masser--W\"ustholz.

For the purpose of open access, the authors have applied a Creative Commons Attribution (CC BY) licence to any author accepted manuscript version arising.

\subsection*{Data availability statement}
Data sharing is not applicable to this article, as no datasets were generated or analysed during the present work.

\section{Preliminaries on lattices and orders}

For further background, we refer the reader to \cite{MR632548,MR1972204,MR1215934}.
Let $\mathcal{O}$ be a Dedekind domain with field of fractions $K$.
To avoid trivialities, we assume that $\mathcal{O} \neq K$.

\subsection{Lattices over Dedekind domains}\label{subsec:lattices-Dedekind}
An \textit{$\mathcal{O}$-lattice} $M$ is a finitely generated torsion-free $\mathcal{O}$-module, or equivalently,
a finitely generated projective $\mathcal{O}$-module.
Using the former definition and the fact that $\mathcal{O}$ is noetherian, we see that 
any $\mathcal{O}$-submodule of an $\mathcal{O}$-lattice is again an $\mathcal{O}$-lattice.

For any finite-dimensional $K$-vector space $V$, an $\mathcal{O}$-lattice in $V$ is a finitely generated $\mathcal{O}$-submodule $M$ of $V$. Given such an $M$, we define a $K$-vector subspace of $V$ by 
\[
K M := \{ \alpha_{1} m_{1} + \alpha_{2} m_{2} + \cdots + \alpha_{r} m_{r} \mid r \in \Z_{\geq 0}, \alpha_i \in K, m_{i} \in M \}
\]
and say that $M$ is a \textit{full} $\mathcal{O}$-lattice in $V$ if $K M=V$. 
Each $\mathcal{O}$-lattice $M$ may be viewed as a full $\mathcal{O}$-lattice in the finite-dimensional 
$K$-vector space $K \otimes_{\mathcal{O}} M$ by identifying $M$ with its image $1 \otimes M$.
We may identify $K \otimes_{\mathcal{O}} M$ with $KM$.

Let $M$ and $N$ be a pair of full $\mathcal{O}$-lattices in a finite-dimensional $K$-vector space
$V$. Since $N$ contains a $K$-basis for $V$, for each $m \in M$ there is a nonzero
$r \in \mathcal{O}$ such that $rm \in N$. 
Therefore there exists a nonzero $r \in \mathcal{O}$ such that $rM \subseteq N$
since $M$ is finitely generated over $\mathcal{O}$. 

For a maximal ideal $\mathfrak{p}$ of $\mathcal{O}$,
let $\mathcal{O}_{\mathfrak{p}}$ denote the localisation of $\mathcal{O}$ at $\mathfrak{p}$.
Let $\widehat{\mathcal{O}}_{\mathfrak{p}}$
denote the completion of $\mathcal{O}$ at $\mathfrak{p}$ and let
$\widehat{K}_{\mathfrak{p}}$ denote its field of fractions. 
For an $\mathcal{O}$-lattice $M$, 
we define the localisation $M$ at $\mathfrak{p}$ to be 
the $\mathcal{O}_{\mathfrak{p}}$-lattice $M_{\mathfrak{p}} := \mathcal{O}_{\mathfrak{p}} \otimes_{\mathcal{O}} M$ and the completion of $M$ at $\mathfrak{p}$ to be 
the $\widehat{\mathcal{O}}_{\mathfrak{p}}$-lattice 
$\widehat{M}_{\mathfrak{p}} := \widehat{\mathcal{O}}_{\mathfrak{p}} \otimes_{\mathcal{O}} M$. 
By identifying $M$ with its image $1 \otimes M$, we may view
$M$ as embedded in $M_{\mathfrak{p}}$.
Viewing $M$ and each $M_{\mathfrak{p}}$ as embedded in $KM$,
we have $M = \bigcap_{\mathfrak{p}} M_{\mathfrak{p}}$, where $\mathfrak{p}$ ranges
over all maximal ideals of $\mathcal{O}$ (see \cite[(4.21)]{MR1972204}).

\subsection{Generalised module indices}\label{subsec:gen-mod-indices}
Much of the material in the following paragraph is explained in more detail in 
\cite[\S 3]{MR236145}.
Let $M,N$ be full $\mathcal{O}$-lattices in a finite-dimensional $K$-vector space $V$.
First consider the case in which $\mathcal{O}$ is a discrete valuation ring. 
Then $M$ and $N$ are both free and of equal rank over $\mathcal{O}$, and so there
exists an $\alpha \in \Aut_{K}(V)$  with $\alpha(M) = N$. Moreover, $\alpha$ is 
unique modulo $\Aut_{\mathcal{O}}(N)$;
hence its determinant is unique modulo $\mathcal{O}^{\times}$, and so the ideal
$[M : N]_{\mathcal{O}} := \mathcal{O} \det(\alpha)$
is a uniquely defined fractional ideal of $\mathcal{O}$.
Now consider the case in which $\mathcal{O}$ is an arbitrary Dedekind domain. 
For almost all maximal ideals $\mathfrak{p}$ of $\mathcal{O}$
we have $M_{\mathfrak{p}} = N_{\mathfrak{p}}$ 
and 
hence 
$[M_{\mathfrak{p}} : N_{\mathfrak{p}} ]_{\mathcal{O}_{\mathfrak{p}}} = \mathcal{O}_{\mathfrak{p}}$
(see \cite[Exercise 4.7]{MR632548}).
Therefore there is a unique fractional ideal $[M:N]_{\mathcal{O}}$ of $\mathcal{O}$ such that
$([M:N]_{\mathcal{O}})_{\mathfrak{p}} = [M_{\mathfrak{p}}:N_{\mathfrak{p}}]_{\mathcal{O}_{\mathfrak{p}}}$
for all $\mathfrak{p}$.
Note that if $M_{1},M_{2},M_{3}$ are full $\mathcal{O}$-lattices in $V$, then
$[M_{1} : M_{3}]_{\mathcal{O}} = [M_{1} :M_{2}]_{\mathcal{O}} \cdot [M_{2}:M_{3}]_{\mathcal{O}}$.
Moreover, if $\mathcal{O}'$ is a Dedekind domain containing $\mathcal{O}$, then 
$\mathcal{O}' \otimes_{\mathcal{O}} [M : N]_{\mathcal{O}}
= [\mathcal{O}' \otimes_{\mathcal{O}} M : \mathcal{O}' \otimes_{\mathcal{O}} N]_{\mathcal{O}'}$.
If $N \subseteq M$ are $\Z$-lattices of equal rank, then we abbreviate $[M:N]_{\Z}$ to $[M:N]$, which is consistent with the fact that $[M:N]_{\Z}$ is the $\Z$-ideal generated by the usual group index of $N$ in $M$.

\begin{lemma}\label{lem:index-diagram}
Suppose we have a diagram of $\mathcal{O}$-lattices with exact rows 
\[
\begin{tikzcd}
0 \ar{r}&\ar{r}  N_{1}  \ar[hookrightarrow]{d} & N_{2} \ar[hookrightarrow]{d}  \ar{r} &\ar{r} N_{3} 
\ar[hookrightarrow]{d} &0\\
0 \ar{r} & M_{1} \ar{r} & M_{2} \ar{r}& \ar{r} M_{3} &0,
\end{tikzcd}
\]
such that $KN_{i} = KM_{i}$ for $i=1,2,3$. 
Then $[M_{2} : N_{2}]_{\mathcal{O}} = [M_{1} : N_{1}]_{\mathcal{O}} \cdot [M_{3} : N_{3}]_{\mathcal{O}}$.
\end{lemma}

\begin{proof}
For $i=1,3$, fix $K$-linear maps $\alpha_{i} \colon KM_{i} \rightarrow KM_{i}$ such that $\alpha_{i}(M_{i})=N_{i}$.
Let $\pi \colon N_{2} \rightarrow N_{3}$ denote the map in the above diagram.
Since $N_{3}$ is $\mathcal{O}$-projective, there exists an $\mathcal{O}$-section
$s \colon N_{3} \rightarrow N_{2}$ of $\pi$.
Then define $\tilde{\alpha}_{3} \colon KM_{3} \rightarrow KM_{2}$ by 
$\tilde{\alpha}_{3} = (K \otimes_{\mathcal{O}} s) \circ \alpha_{3}$. 
Fixing an $\mathcal{O}$-linear splitting $M_{2} \cong M_{1} \oplus M_{3}$ 
(which exists since $M_{3}$ is $\mathcal{O}$-projective) and thus a $K$-linear splitting 
$KM_{2} \cong KM_{1} \oplus KM_{3}$, we then obtain a $K$-linear map 
$\alpha_{2} := (\alpha_{1} + \tilde{\alpha}_{3}) \colon  KM_{2} \rightarrow KM_{2}$ such that 
$\alpha_{2}(M_{2})=N_{2}$ and $\alpha_{1}(M_{1})=N_{1}$. 
Hence, with respect to a $K$-basis of $KM_{2}$ extending a $K$-basis of $KM_{1}$, the matrix representing
$\alpha_{2}$ is block upper triangular.
Consequently, $\det(\alpha_{2})=\det(\alpha_{1})\det(\alpha_{3})$, and thus we obtain the desired result.
\end{proof}

\subsection{Duals of lattices}\label{subsec:duals-of-lattices}
Let $M$ be an $\mathcal{O}$-lattice. 
The linear dual $M^{\vee} := \Hom_{\mathcal{O}}(X,\mathcal{O})$ is also an $\mathcal{O}$-lattice
and there is a canonical identification $(M^{\vee})^{\vee}=M$. 
Moreover, $(-)^{\vee}$ is inclusion-reversing.
For a maximal ideal $\mathfrak{p}$ of $\mathcal{O}$, we have 
$(M_{\mathfrak{p}})^{\vee} = (M^{\vee})_{\mathfrak{p}}$.
Together with the fact that determinants are invariant under transposition, this implies that
if $M$ and $N$ are full $\mathcal{O}$-lattices in a finite-dimensional $K$-vector space $V$,
then $[M : N]_{\mathcal{O}} = [N^{\vee} : M^{\vee}]_{\mathcal{O}}$.

\subsection{Lattices over orders}
Let $A$ be a finite-dimensional $K$-algebra and let $\Lambda$ be an \textit{$\mathcal{O}$-order} in $A$, that is, a subring of $A$ that is also a full $\mathcal{O}$-lattice in $A$. 
Note that $\Lambda$ is both left and right noetherian since $\Lambda$ is finitely generated over
$\mathcal{O}$. 
A left \textit{$\Lambda$-lattice} $X$ is a left $\Lambda$-module that when considered as an $\mathcal{O}$-module is also an $\mathcal{O}$-lattice; in this case, $KX$ may be viewed as a left $A$-module.

Henceforth all modules (resp.\ lattices) shall be assumed to be left modules (resp.\ lattices) unless otherwise stated.
Two $\Lambda$-lattices are said to be isomorphic if they are isomorphic as $\Lambda$-modules.

For a maximal ideal $\mathfrak{p}$ of $\mathcal{O}$,
the localisation $\Lambda_{\mathfrak{p}}$ is an 
$\mathcal{O}_{\mathfrak{p}}$-order in $A$, and 
 the completion $\widehat{\Lambda}_{\mathfrak{p}}$ is a 
$\widehat{\mathcal{O}}_{\mathfrak{p}}$-order
in $\widehat{K}_{\mathfrak{p}} \otimes_{K} A$.
Localising a $\Lambda$-lattice $X$ 
at $\mathfrak{p}$ yields
a $\Lambda_{\mathfrak{p}}$-lattice $X_{\mathfrak{p}}$, 
and completing $X$ at $\mathfrak{p}$ yields 
a $\widehat{\Lambda}_{\mathfrak{p}}$-lattice $\widehat{X}_{\mathfrak{p}}$.
Given  $\Lambda$-lattices $X$ and $Y$,
we have that $X_{\mathfrak{p}} \cong Y_{\mathfrak{p}}$ as $\Lambda_{\mathfrak{p}}$-lattices
if and only if $\widehat{X}_{\mathfrak{p}} \cong \widehat{Y}_{\mathfrak{p}}$ as 
$\widehat{\Lambda}_{\mathfrak{p}}$-lattices (see \cite[(18.2)]{MR1972204}).
For a positive integer $n$, a $\Lambda$-lattice $X$ is said to be \textit{locally free of rank $n$},
if for each maximal ideal $\mathfrak{p}$ of $\mathcal{O}$,
the $\Lambda_{\mathfrak{p}}$-lattice $X_{\mathfrak{p}}$ is free of rank $n$,  
or equivalently, the 
$\widehat{\Lambda}_{\mathfrak{p}}$-lattice $\widehat{X}_{\mathfrak{p}}$ is free of rank $n$. 
Note that every locally free $\Lambda$-lattice is projective by \cite[(8.19)]{MR632548}. 

\subsection{Maximal orders}\label{subsec:max-orders}
Suppose that $A$ is a separable finite-dimensional $K$-algebra (see \cite[\S 7c]{MR1972204}).
A \textit{maximal} $\mathcal{O}$-order in $A$ is an $\mathcal{O}$-order that is not properly contained in any other $\mathcal{O}$-order in $A$.
For any $\mathcal{O}$-order $\Lambda$ in $A$, there exists a (not necessarily unique)
maximal $\mathcal{O}$-order $\mathcal{M}$ in $A$ containing $\Lambda$ 
by \cite[(10.4)]{MR1972204}.
If $\mathcal{M}$ is a maximal $\mathcal{O}$-order,  
$X$ is an $\mathcal{M}$-lattice, and $n$ is a positive integer, 
then by \cite[(31.2)(iii)]{MR632548} we have that
$KX$ is free of rank $n$ over $A$ if and only if $X$ is locally free of rank $n$.

\subsection{Locally free class groups and cancellation properties}\label{subsec:lfcg-sfc}
Suppose that $K$ is a global field and that $A$ is separable finite-dimensional $K$-algebra.
Let $\Lambda$ be an $\mathcal{O}$-order $A$.
Let $P(\Lambda)$ be the free abelian group generated by symbols $[X]$, one for each 
isomorphism class of locally free $\Lambda$-lattices $X$, modulo relations 
$[X] = [X_{1}] + [X_{2}]$ whenever $X \cong X_{1} \oplus X_{2}$.
We define the \textit{locally free class group} $\Cl(\Lambda)$ of $\Lambda$ 
to be the subgroup of $P(\Lambda)$ consisting of all elements that can be written in the form 
$[X]-[Y]$, with $X,Y$ locally free and $KX \cong KY$. 

We remark that $[X]-[Y] =0$ in $\Cl(\Lambda)$ if and only if $X$ is \textit{stably isomorphic to} $Y$,
that is, $X \oplus \Lambda^{(k)} \cong Y \oplus \Lambda^{(k)}$ for some positive integer $k$
(here $\Lambda^{(k)}$ denotes the direct sum of $k$ copies of $\Lambda$).
The order $\Lambda$ is said to have the \textit{locally free cancellation property} if given any pair of 
locally free $\Lambda$-lattices $X$ and $Y$,
\[
X \oplus \Lambda^{(k)} \cong Y \oplus \Lambda^{(k)} \textrm{ for some } k \in \Z_{\geq 0} \implies X \cong Y.
\]
Moreover, $\Lambda$ is said to have the
\textit{stably free cancellation property} 
if this holds in the special case that $Y$ is free.
If $A$ satisfies the so-called \textit{Eichler condition} relative to $\mathcal{O}$, then $\Lambda$ has the
locally free cancellation property; this condition is satisfied 
if $A$ is commutative (see \cite[\S 51]{MR892316} for further details).

If $n$ is a positive integer and
$Y$ is any locally free $\Lambda$-lattice of rank $n$, then by \cite[(31.14)]{MR632548} there exists a locally free $\Lambda$-lattice $X$ in $A$ such that $Y \cong \Lambda^{(n-1)} \oplus X$. 
Hence every element of $\Cl(\Lambda)$ is expressible in the form $[X_{1}]-[X_{2}]$,
where $X_{1}$ and $X_{2}$ are locally free $\Lambda$-lattices in $A$.
In particular, the Jordan--Zassenhaus theorem \cite[(26.4)]{MR1972204} 
implies that $\Cl(\Lambda)$ is finite.
Moreover, for each such pair $X_{1}, X_{2}$, 
there exists another such lattice $X_{3}$ such that $X_{2} \oplus X_{3} \cong \Lambda \oplus X_{1}$
by \cite[(31.7)]{MR632548}.
Therefore every element of $\Cl(\Lambda)$ is in fact expressible in the form $[X]-[\Lambda]$
for some locally free $\Lambda$-lattice $X$ in $A$.

The following result is easily deduced from the above discussion.

\begin{lemma}\label{lem:equiv-locally-free-implies-free}
The following statements are equivalent:
\begin{enumerate}
\item every locally free $\Lambda$-lattice is in fact free;
\item every locally free $\Lambda$-lattice of rank $1$ is in fact free;
\item $\Cl(\Lambda)$ is trivial and $\Lambda$ has the stably free cancellation property;
\item $\Cl(\Lambda)$ is trivial and $\Lambda$ has the locally free cancellation property.
\end{enumerate} 
\end{lemma}

\begin{remark}\label{rmk:cancellation-max-order}
Lemma \ref{lem:equiv-locally-free-implies-free} will often be applied in the case that $\Lambda$
is a maximal $\Z$-order. Smertnig and Voight \cite[Theorem 1.3]{MR4105795} have classified
all maximal $\Z$-orders in totally definite quaternion algebras with the locally free cancellation 
property, and from this it is straightforward to determine whether any given maximal $\Z$-order
in a finite-dimensional semisimple $\Q$-algebra has the locally free cancellation property. 
See also Remark \ref{rmk:group-alg-max-order-canc-prop}.
\end{remark}

\section{Overlattices and sublattices for overorders}\label{sec:over-and-sublattices}

Let $\OO$ be a Dedekind domain with field of fractions $K$ and assume that $\mathcal{O} \neq K$.

\subsection{Setup and definitions}
Let $A$ be a finite-dimensional $K$-algebra.
Let $\Lambda \subseteq \Gamma$ be $\mathcal{O}$-orders in $A$ and let
$X$ be a $\Lambda$-lattice. Define 
\[
\Gamma X := \{ \gamma_{1} m_{1} + \gamma_{2} m_{2} + \cdots + \gamma_{r} m_{r} \mid r \in \Z_{\geq 0}, m_{i} \in X, \gamma_{i} \in \Gamma \} \subseteq KX.
\]
This is the unique $\Gamma$-lattice in $KX$
containing $X$ that is minimal with respect to inclusion.

There exists a nonzero $r \in \mathcal{O}$ such that $r\Gamma \subseteq \Lambda$
(see \S \ref{subsec:lattices-Dedekind})
and so  $r\Gamma X$ is a $\Gamma$-lattice contained in $X$ of finite index.
Since the sum of any two $\Gamma$-lattices contained in $X$
is also a $\Gamma$-lattice contained in $X$, we see that there exists a 
unique $\Gamma$-lattice contained in $X$ that is maximal with respect to inclusion, 
which we shall denote by $\leftindex^{\Gamma}{X}$.

For a right $\Lambda$-lattice $X$, we define $X\Gamma$ and $X^{\Gamma}$ similarly. 
Note that $\leftindex^{\Gamma}{\Lambda}$ (resp.\ $\Lambda^{\Gamma}$) coincides with the 
right (resp.\ left) conductor of $\Gamma$ into 
$\Lambda$ (see \cite[(27.2)]{MR632548}).

\subsection{Bounds on indices}
The following result gives a bound on 
\[
[\Gamma X : \leftindex^{\Gamma}{X}]_{\mathcal{O}}
= [\Gamma X : X]_{\mathcal{O}} \cdot [X :  \leftindex^{\Gamma}{X}]_{\mathcal{O}}
\]
that only depends on $\Gamma$ and $\Lambda$, and not on the particular choice of lattice $X$.

\begin{prop}\label{prop:bddbyconductor}
Let $A$ be a finite-dimensional $K$-algebra and let $\Lambda \subseteq \Gamma$ be $\mathcal{O}$-orders in $A$.  Let $J$ be any full two-sided ideal of $\Gamma$ contained in $\Lambda$.
Let $X$ be a $\Lambda$-lattice such that $\Gamma X$
is locally free of rank $n$ over $\Gamma$.
Then $[\Gamma X : \leftindex^{\Gamma}{X}]_{\mathcal{O}}$ 
divides $[\Gamma : J]_{\mathcal{O}}^{n}$.
\end{prop}

\begin{remark}
There are many possible choices of $J$, and the best choice will be context specific.
For example, a weak but general choice is $J=[\Gamma:\Lambda]_{\mathcal{O}} \cdot \Gamma$.
Moreover, $J$ can always be taken to be the two-sided ideal of $\Gamma$ generated by the central conductor of $\Gamma$ into $\Lambda$, that is, by $\{ x \in C \mid x\Gamma \subseteq \Lambda \}$, where $C$ denotes
the centre of $A$.
\end{remark}

\begin{proof}[Proof of Proposition \ref{prop:bddbyconductor}]
Since $J$ is a left $\Gamma$-lattice contained in $\Lambda$, 
we have that $JX$ is a left $\Gamma$-lattice contained in $X$. 
Hence $JX$ is contained in $\leftindex^{\Gamma}{X}$.
The chain of containments
\[
JX \subseteq \leftindex^{\Gamma}{X} \subseteq  X \subseteq \Gamma X 
\]
implies that
$[\Gamma X : \leftindex^{\Gamma}{X} ]_{\mathcal{O}}$
divides $[\Gamma X : JX]_{\mathcal{O}}$. 
Thus it remains to show that
\[
[\Gamma X : J X]_{\mathcal{O}}=[\Gamma : J]_{\mathcal{O}}^{n}.
\] 
Since indices are defined locally and 
$([\Gamma X : JX]_{\mathcal{O}})_{\mathfrak{p}}
=[\Gamma_{\mathfrak{p}} X_{\mathfrak{p}} : J_{\mathfrak{p}}X_{\mathfrak{p}}]_{\mathcal{O}_{\mathfrak{p}}}$
for every maximal ideal $\mathfrak{p}$ of $\mathcal{O}$, we can and do assume without loss of generality 
that $\mathcal{O}$ is a discrete valuation ring. 
Then by hypothesis there exist $\varepsilon_{1}, \ldots, \varepsilon_{n}\in \Gamma X$ such that 
$\Gamma X = \Gamma \varepsilon_{1} \oplus \cdots \oplus \Gamma \varepsilon_{n}$.
Since $J$ is a right $\Gamma$-module, we have
\[
JX
=J\Gamma X
=J(\Gamma \varepsilon_{1} \oplus \cdots \oplus \Gamma \varepsilon_{n})
=J \varepsilon_{1} \oplus \cdots \oplus J \varepsilon_{n}.
\]
Therefore  
\[
[\Gamma X : J  X]_{\mathcal{O}}
= [\Gamma \varepsilon_{1} \oplus \cdots \oplus \Gamma \varepsilon_{n} : J \varepsilon_{1} \oplus \cdots \oplus J \varepsilon_{n}]_{\mathcal{O}}
=[\Gamma : J]_{\mathcal{O}}^{n}. \qedhere
\]
\end{proof} 

\begin{corollary}\label{cor:bddbyconductor}
Let $A$ be a separable finite-dimensional $K$-algebra and let $\Lambda$ be an $\mathcal{O}$-order in $A$.
Let $\mathcal{M}$ be a maximal $\mathcal{O}$-order in $A$ containing $\Lambda$ and
let $J$ be any full two-sided ideal of $\mathcal{M}$ contained in $\Lambda$.
Let $X$ be a $\Lambda$-lattice such that $KX$ is free of rank $n$ over $A$. 
Then $[\mathcal{M} X : \leftindex^{\mathcal{M}}{X}]_{\mathcal{O}}$ 
divides $[\mathcal{M} : J]_{\mathcal{O}}^{n}$.
\end{corollary}

\begin{proof}
Since $KX$ is free of rank $n$ over $A$, we have that
$\mathcal{M}X$ is locally free of rank $n$ over $ \mathcal{M}$ (see \S \ref{subsec:max-orders}),
and so the desired result follows directly from Proposition \ref{prop:bddbyconductor}.
\end{proof}

\begin{remark}
In Proposition \ref{prop:bddbyconductor} and Corollary \ref{cor:bddbyconductor}, the order $\Lambda$
can be replaced by the so-called \textit{associated order} $\mathcal{A}(X) = \{ \alpha \in A \mid \alpha X \subseteq X\}$.
Thus if the containment $\Lambda \subseteq \mathcal{A}(X)$ is strict, then working over $\mathcal{A}(X)$
may allow a choice of $J \subseteq \mathcal{A}(X)$ with improved index
$[\mathcal{M} : J]_{\mathcal{O}}$.
For example, if $\mathcal{M}$ is a maximal order containing $\Lambda$ and we take $X=\mathcal{M}$,
then $\mathcal{A}(X)=\mathcal{M}$ and so we can take $J=\mathcal{M}$, which is consistent with the fact
that $\mathcal{M} X = X = \leftindex^{\mathcal{M}}{X}$ in this situation.
Of course, the disadvantage of this approach is that $\mathcal{A}(X)$ depends on $X$.
\end{remark}

\subsection{Duals and overorders}\label{subsec:duals-and-overorders}

For an $\mathcal{O}$-order $\Lambda$ in a finite-dimensional $K$-algebra and any 
left (resp.\ right) $\Lambda$-lattice $X$, the dual 
$X^{\vee}=\Hom_{\mathcal{O}}(X,\mathcal{O})$ 
has the structure of a right (resp.\ left) $\Lambda$-lattice,
and there is a canonical identification $(X^{\vee})^{\vee}=X$.

\begin{lemma}\label{lem:dualsubover}
Let $\Lambda \subseteq \Gamma$ be $\mathcal{O}$-orders in a finite-dimensional $K$-algebra.
\begin{enumerate}
\item If $X$ is a left $\Lambda$-lattice, then 
we have an equality of right $\Gamma$-lattices
$(\Gamma X)^{\vee} = (X^{\vee})^{\Gamma}$ 
and an equality of indices 
$[\Gamma X : X]_{\mathcal{O}} = [X^{\vee} : (X^{\vee})^{\Gamma}]_{\mathcal{O}}$. 
\item If $X$ is a right $\Lambda$-lattice, then 
we have an equality of left $\Gamma$-lattices $(X\Gamma)^{\vee}=\lGamma{(X^{\vee})}$
and an equality of indices $[X\Gamma : X]_{\mathcal{O}}= [X^{\vee} : \lGamma{(X^{\vee})}]_{\mathcal{O}}$.
\end{enumerate}
\end{lemma}

\begin{proof}
We only prove part (i).
Since $(-)^{\vee}$ reverses inclusions,
$(\Gamma X)^{\vee}$ is a right $\Gamma$-lattice contained in $X^{\vee}$.
Hence $(\Gamma X)^{\vee}$ is contained in $(X^{\vee})^{\Gamma}$ by definition of the latter.	
Dualising, we also have that 
\begin{equation}\label{eq:dualisinglattices}  
\Gamma X=((\Gamma X)^{\vee})^{\vee} \supseteq ((X^{\vee})^{\Gamma})^{\vee}.
\end{equation}
Since $((X^{\vee})^{\Gamma})^{\vee}$ is itself a left $\Gamma$-lattice containing $X$, this forces equality in \eqref{eq:dualisinglattices} and hence $(\Gamma X)^{\vee}=(X^{\vee})^{\Gamma}$ 
as desired. 
Finally, since $(-)^{\vee}$ preserves indices (see \S \ref{subsec:duals-of-lattices}) we have that
$[\Gamma X: X]_{\mathcal{O}}
= [ X^{\vee} : (\Gamma X)^{\vee}]_{\mathcal{O}}
= [ X^{\vee} : (X^{\vee})^{\Gamma}]_{\mathcal{O}}$.
\end{proof}

\subsection{The commutative separable setting}
In the setting of commutative separable algebras, the following result of Fr{\"o}hlich
is a refinement of Corollary \ref{cor:bddbyconductor}.

\begin{theorem}[Fr\"ohlich \cite{MR0210697}]\label{thm:frohlich-commutative}
Let $A$ be a commutative separable finite-dimensional $K$-algebra and let
$\Lambda$ be an $\mathcal{O}$-order in $A$. 
Let $\mathcal{M}$ be the unique maximal $\mathcal{O}$-order in $A$.
Let $X$ be a $\Lambda$-lattice such that $KX$ is free of rank $n$ over $A$. 
Then both $[\mathcal{M} X : X]_{\mathcal{O}}$ and 
$[X : \leftindex^{\mathcal{M}}{X}]_{\mathcal{O}}$ divide $[\mathcal{M} : \Lambda]_{\mathcal{O}}^{n}$. 
\end{theorem}

\begin{proof}
The claim that $[\mathcal{M} X : X]_{\mathcal{O}}$ divides $[\mathcal{M} : \Lambda]_{\mathcal{O}}^{n}$ is contained in \cite[Theorem 4]{MR0210697}. 

It remains to show that $[X : \leftindex^{\mathcal{M}}{X}]_{\mathcal{O}}$ divides $[\mathcal{M} : \Lambda]_{\mathcal{O}}^{n}$.
Since $A$ is separable there is an isomorphism of (right) $A$-modules $A \cong \Hom_{K}(A,K)$
induced by the pairing of \cite[(7.41)]{MR632548}.
Thus there are $A$-module isomorphisms
\[
K(X^{\vee}) \cong \Hom_{K}(KX,K) \cong \Hom_{K}(A^{(n)},K)  \cong \Hom_{K}(A,K)^{(n)} \cong A^{(n)}.
\]
Lemma \ref{lem:dualsubover}(ii) implies that
$[X : \leftindex^{\mathcal{M}}{X}]_{\mathcal{O}}
= [X^{\vee}\mathcal{M} : X^{\vee}]_{\mathcal{O}} =  [\mathcal{M} X^{\vee}: X^{\vee}]_{\mathcal{O}}$,
where in the last equality, we consider $X^{\vee}$ as a left $\mathcal{M}$-lattice, as we may since
$\mathcal{M}$ is commutative.
Moreover, since $K(X^{\vee})$ is free of rank $n$ over $A$, the first claim and the appropriate 
substitution imply that $[\mathcal{M} X^{\vee}: X^{\vee}]_{\mathcal{O}}$ divides 
$[\mathcal{M} : \Lambda]_{\mathcal{O}}^{n}$.
\end{proof}

\begin{remark}
\cite[\S 7, Example 1]{MR274497} shows that the result analogous to Theorem \ref{thm:frohlich-commutative} does not always hold in the noncommutative separable setting.
\end{remark}

\section{Free sublattices of bounded index}\label{sec:free-sublattices-of-bounded-index}

Let $\mathcal{O}$ be a Dedekind domain with field of fractions $K$.
Assume that $K$ is a global field and that $\mathcal{O} \neq K$.
Let $A$ be a separable finite-dimensional $K$-algebra. 

\subsection{Free sublattices of locally free lattices}
The following result gives a bound on the index of a free sublattice in a locally free lattice.

\begin{prop}\label{prop:free-sublattice-of-lf}
Let $\Gamma$ be an $\mathcal{O}$-order in $A$ and 
let $\mathcal{K}$ be any nonzero ideal of $\mathcal{O}$. 
Then there exists a nonzero ideal $\mathcal{I}$ of $\mathcal{O}$,
that can be chosen to be coprime to $\mathcal{K}$, with the following property:
for every locally free $\Gamma$-lattice $X$,
there exists a free $\Gamma$-sublattice $Y$ of $X$ 
such that $[X:Y]_{\mathcal{O}}$ divides $\mathcal{I}$.
\end{prop}

\begin{proof}
By \cite[(31.14)]{MR632548}, for a positive integer $n$ and a locally free $\Gamma$-lattice $X$
of rank $n$, there exists a locally free $\Gamma$-lattice $W$ of rank $1$
such that $X \cong \Gamma^{(n-1)} \oplus W$. 
Thus the problem reduces to the case of locally free $\Lambda$-lattices $W$ of rank $1$.
The number of isomorphism classes of such lattices is 
finite by the Jordan--Zassenhaus theorem \cite[(26.4)]{MR1972204}. 
For each such class, choose a representative $W$ and note that 
by Roiter's lemma \cite[(31.6)]{MR632548} there exists an embedding 
$\iota_{W} : \Gamma \hookrightarrow W$ such that 
$[W : \iota_{W}(\Gamma)]_{\mathcal{O}}$ is coprime to $\mathcal{K}$.
Now take $\mathcal{I}$ to be any common multiple of the (finite number of) ideals 
$[W : \iota_{W}(\Gamma)]_{\mathcal{O}}$ as $W$ varies.
\end{proof}

\begin{corollary}\label{cor:free-sublattice-of-lf-max-case}
Let $\mathcal{M}$ be a maximal $\mathcal{O}$-order in $A$
and let $\mathcal{K}$ be any nonzero ideal of $\mathcal{O}$. 
Then there exists a nonzero ideal $\mathcal{I}$ of $\mathcal{O}$,
that can be chosen to be coprime to $\mathcal{K}$, with the following property:
given any $\mathcal{M}$-lattice $X$ such that $KX$ is free as an $A$-module,
there exists a free $\mathcal{M}$-sublattice $Y$ of $X$ 
such that $[X:Y]_{\mathcal{O}}$ divides $\mathcal{I}$. 
\end{corollary}

\begin{proof}
Let $X$ be an $\mathcal{M}$-lattice. 
Then $KX$ is free as an $A$-module if and only if 
$\mathcal{M}X$ is locally free over $ \mathcal{M}$ 
(see \S \ref{subsec:max-orders}).
Hence the result follows from Proposition \ref{prop:free-sublattice-of-lf}.
\end{proof}

\begin{remark}\label{rmk:I-trivial}
If $\Gamma$ (resp.\ $\mathcal{M}$) satisfies the equivalent conditions of 
Lemma \ref{lem:equiv-locally-free-implies-free},
then it is clear that we can take $\mathcal{I}=\mathcal{O}$ in 
Proposition \ref{prop:free-sublattice-of-lf} (resp.\ Corollary \ref{cor:free-sublattice-of-lf-max-case}).
\end{remark}

Given a finite set $S$ of maximal ideals of $\mathcal{O}$, let 
$\mathcal{O}_{S} = \bigcap_{\mathfrak{p} \notin S} \mathcal{O}_{\mathfrak{p}}$,
where $\mathfrak{p}$ ranges over all maximal ideals of $\mathcal{O}$ not in $S$.
We include the following result for general interest. 

\begin{corollary}\label{cor:choice-of-S}
Let $\Gamma$ be an $\mathcal{O}$-order in $A$ and 
let $T$ be a finite set of maximal ideals of $\mathcal{O}$.
Then there exists a finite set $S$ of maximal ideals of $\mathcal{O}$ such that 
$S \cap T = \emptyset$
and $\mathcal{O}_{S} \otimes_{\mathcal{O}} \Gamma$ 
satisfies the equivalent conditions of Lemma \ref{lem:equiv-locally-free-implies-free}.
\end{corollary}

\begin{proof}
If $\mathcal{K}$ is the product of the maximal ideals in $T$ and
$\mathcal{I}$ is the ideal given by Proposition~\ref{prop:free-sublattice-of-lf},
then we can take $S$ to be the set of maximal ideals dividing $\mathcal{I}$.
\end{proof}

\subsection{The main theorem}
The main theorem of the present article is as follows.

\begin{theorem}\label{thm:main}
Let $\mathcal{O}$ be a Dedekind domain with field of fractions $K$. 
Assume that $K$ is a global field and that $\mathcal{O} \neq K$.
Let $A$ be a separable finite-dimensional $K$-algebra and let $\Lambda$ be an $\mathcal{O}$-order in $A$.
Let $\mathcal{M}$ be a maximal $\mathcal{O}$-order in $A$ containing $\Lambda$ and
let $J$ be any full two-sided ideal of $\mathcal{M}$ contained in $\Lambda$.
Let $\mathcal{K}$ be any nonzero ideal of $\mathcal{O}$.
Then there exists a nonzero ideal $\mathcal{I}$ of $\mathcal{O}$,
that can be chosen to be coprime to $\mathcal{K}$, with the following property:
given any $\Lambda$-lattice $X$  such that $KX$ is free of rank $n$ over $A$,
there exists a free $\Lambda$-sublattice $Z$ of $X$ such that 
$[X : Z]_{\mathcal{O}}$ divides 
$\mathcal{I} \cdot [\mathcal{M} : \Lambda]_{\mathcal{O}}^{2n}$
if $A$ is commutative or
$\mathcal{I} \cdot [\mathcal{M} : J]_{\mathcal{O}}^{n} 
\cdot [\mathcal{M} : \Lambda]_{\mathcal{O}}^{n}$ otherwise.
Moreover, if $\mathcal{M}$ satisfies the equivalent conditions of 
Lemma \ref{lem:equiv-locally-free-implies-free}, then we can take $\mathcal{I}=\mathcal{O}$.
\end{theorem}

\begin{proof}
Let $\mathcal{I}$ be the ideal of $\mathcal{O}$ given by Corollary \ref{cor:free-sublattice-of-lf-max-case}.
If $\mathcal{M}$ satisfies the equivalent conditions of Lemma \ref{lem:equiv-locally-free-implies-free},
then we can take  $\mathcal{I}=\mathcal{O}$ by Remark \ref{rmk:I-trivial}.
Then there exists a free $\mathcal{M}$-sublattice $Y$ of $\leftindex^{\mathcal{M}}{X}$
such that $[\leftindex^{\mathcal{M}}{X} : Y]_{\mathcal{O}}$ divides $\mathcal{I}$.
Let $\varepsilon_{1}, \ldots, \varepsilon_{n}$ be a free $\mathcal{M}$-basis of $Y$,
so that $Y=\mathcal{M} \varepsilon_{1} \oplus \cdots \oplus \mathcal{M} \varepsilon_{n}$,
and let $Z=\Lambda \varepsilon_{1} \oplus \cdots \oplus \Lambda \varepsilon_{n}$.
Then 
$Z \subseteq Y \subseteq \leftindex^{\mathcal{M}}{X} \subseteq X$ and
$[X: Z]_{\mathcal{O}}
=
[X : \leftindex^{\mathcal{M}}{X}]_{\mathcal{O}}
\cdot [ \leftindex^{\mathcal{M}}{X} : Y]_{\mathcal{O}} \cdot [Y:Z]_{\mathcal{O}}$.
Note that $[Y:Z]_{\mathcal{O}} =  [\mathcal{M} : \Lambda]_{\mathcal{O}}^{n}$.
Moreover, Corollary \ref{cor:bddbyconductor} implies that $[X :  \leftindex^{\mathcal{M}}{X}]_{\mathcal{O}}$ divides $[\mathcal{M} : J]_{\mathcal{O}}^{n}$, and 
under the assumption that $A$ is commutative, Theorem \ref{thm:frohlich-commutative}
implies that in fact $[X :  \leftindex^{\mathcal{M}}{X}]_{\mathcal{O}}$ divides $[\mathcal{M} : \Lambda]_{\mathcal{O}}^{n}$.
Therefore we obtain the desired result.
\end{proof}

\begin{remark}\label{rmk:non-free}
The statement of Theorem \ref{thm:main} extends to $\Lambda$-lattices $X$ admitting a surjection
$A^{(n)} \twoheadrightarrow KX$ of $A$-modules.
More specifically, the ideal $\mathcal{I}$ has the following property: 
given any 
$\Lambda$-lattice $X$ admitting a surjection
$A^{(n)} \twoheadrightarrow KX$ of $A$-modules,
there exists a $\Lambda$-sublattice $Z$ of $X$ generated by at most $n$ elements such that 
$[X : Z]_{\mathcal{O}}$ divides \linebreak[4]
$\mathcal{I} \cdot [\mathcal{M} : J]_{\mathcal{O}}^{n} 
\cdot [\mathcal{M} : \Lambda]_{\mathcal{O}}^{n}$.
This can be seen as follows. 
There exists an $A$-module $B$ such that $KX \oplus B \cong A^{(n)}$. 
Thus given any full $\Lambda$-lattice $W$ in $B$, the $\Lambda$-lattice $X \oplus W$ satisfies the 
conditions of Theorem \ref{thm:main} and so admits a free $\Lambda$-sublattice $Z'$ of index dividing
$\mathcal{I} \cdot [\mathcal{M} : J]_{\mathcal{O}}^{n} 
\cdot [\mathcal{M} : \Lambda]_{\mathcal{O}}^{n}$, and the image of $Z'$ under the projection $X \oplus W \twoheadrightarrow X$ is the desired sublattice $Z$.
Of course, one should expect stronger bounds if one specifies the isomorphism class of $KX$; 
one such situation is considered in \S \ref{sec:group-rings-modulo-trace}.
\end{remark}

\section{Group rings}\label{sec:group-rings}

\subsection{Conductors of group rings}
The extra structure of group rings is exploited in the following result,
which will allow us to make an optimal choice of the two-sided ideal $J$ that appears 
in the statement of Theorem \ref{thm:main}.

\begin{prop}\label{prop:index-in-two-orders}
Let $\mathcal{O}$ be a Dedekind domain with field of fractions $K\neq \OO$. 
Let $G$ be a finite group and let $\Gamma$ be an $\mathcal{O}$-order in $K[G]$ containing
$\mathcal{O}[G]$.
Then 
${\mathcal{O}[G]}^{\Gamma} = \leftindex^{\Gamma}{\mathcal{O}[G]}$ 
and we have
\begin{equation}\label{eq:index-max-eq}
[\Gamma : \mathcal{O}[G]]_{\mathcal{O}} 
= [\mathcal{O}[G] : \leftindex^{\Gamma}{\mathcal{O}[G]}]_{\mathcal{O}}
= [\Gamma : \leftindex^{\Gamma}{\mathcal{O}[G]} ]_{\mathcal{O}}^{\frac{1}{2}}. 
\end{equation}
Moreover, if $|G|$ is invertible in $K$ and $\Gamma=\MM$ is a maximal $\OO$-order, then this index is independent of the choice of $\mathcal{M}$.
\end{prop}

\begin{remark}
In the case that $|G|$ is invertible in $K$ and
$\Gamma=\mathcal{M}$ is a maximal $\mathcal{O}$-order,
Jacobinski has given an explicit description of
${\mathcal{O}[G]}^{\Gamma} = \leftindex^{\Gamma}{\mathcal{O}[G]}$ 
(see Theorem \ref{thm:Jacobinski}) and this leads to an
explicit formula for the index of \eqref{eq:index-max-eq} (see Corollary \ref{cor:index-formula}).
\end{remark}

\begin{proof}[Proof of Proposition \ref{prop:index-in-two-orders}]
Given an $\mathcal{O}$-order $\Lambda$ in $K[G]$,
let $\Lambda^{\opp}$ denote the $\mathcal{O}$-order defined by the
image of $\Lambda$ under the involution on $K[G]$
induced by $g \mapsto g^{-1}$.
Any left (resp.\ right) $\Lambda$-lattice carries a canonical structure of a right (resp.\ left) 
$\Lambda^{\opp}$-lattice with $g^{-1}$ acting as $g$ did previously.
Given a left (resp.\ right) $\Lambda$-lattice $X$, we denote by $X^{*}$ the dual 
lattice $X^{\vee} = \Hom_{\mathcal{O}}(X, \mathcal{O})$ considered as a left (resp.\ right) $\Lambda^{\opp}$-lattice. 
Note that for a left $\Lambda$-lattice $X$, we have 
$[(X^{\vee})\Lambda : X^{\vee}]_{\mathcal{O}} = [\Lambda X^{*} : X^{*}]_{\mathcal{O}}$, etc.

Now observe that $\Gamma^{\opp}$ is an $\mathcal{O}$-order containing 
$\mathcal{O}[G]=\mathcal{O}[G]^{\opp}$.
Hence
$\Gamma^{\opp}\mathcal{O}[G]
=\Gamma^{\opp}
=\mathcal{O}[G]\Gamma^{\opp}$.
We also have that
\begin{align*}
(\Gamma^{\opp}\mathcal{O}[G])^{\vee}&=(\mathcal{O}[G]^{\vee})^{\Gamma^{\opp}}= \leftindex^{\Gamma}{(\mathcal{O}[G]}^{*}), \\
(\mathcal{O}[G]\Gamma^{\opp})^{\vee}&=\leftindex^{\Gamma^{{\opp}}}{(\OO[G]^{\vee})}=(\mathcal{O}[G]^{*})^{\Gamma}, 
\end{align*}
where in each case the first equality follows from 
Lemma \ref{lem:dualsubover} and the second equality follows from the definition of $(-)^{*}$. 
Therefore $\leftindex^{\Gamma}{(\mathcal{O}[G]}^{*}) = (\mathcal{O}[G]^{*})^{\Gamma}$.
Furthermore, there is an $\mathcal{O}[G]=\mathcal{O}[G]^{\opp}$-isomorphism 
$\mathcal{O}[G]^{*} \overset{\sim}{\to} \mathcal{O}[G]$ given by $\one_{g} \mapsto g$, where $\one_{g}$ denotes the element of $\Hom_{\OO}(\OO[G],\OO)$ 
defined by $h \mapsto 0$ for $h\neq g$ and $g \mapsto 1$. 
Hence we conclude that $\leftindex^{\Gamma}{\mathcal{O}[G]} = \mathcal{O}[G]^{\Gamma}$.	

We have that 
$[ \Gamma : \mathcal{O}[G]]_{\mathcal{O}} 
= [ \Gamma^{\opp} : \mathcal{O}[G]^{\opp}]_{\mathcal{O}}$
since $(-)^{\opp}$ is an $\mathcal{O}$-linear isomorphism.
As $\OO[G]=\OO[G]^{\opp}$, we then have
\begin{align*} 
[ \Gamma : \OO[G]]_\OO&= 
[ \Gamma^{\opp} : \OO[G]]_\OO \\
&=[(\OO[G])^{\vee}: ((\OO[G])^\vee)^{\Gamma^\opp} ]_\OO\\
&=[ \OO[G]^*:\leftindex^{\Gamma}{((\OO[G])^*)}]_\OO\\
 &=[\OO[G]: \leftindex^{\Gamma}{\OO[G]}]_\OO,
\end{align*}
where the second equality follows from Lemma \ref{lem:dualsubover}(i).
Since 
\[
[\Gamma : \leftindex^{\Gamma}{\mathcal{O}[G]} ]_{\mathcal{O}}
= [\Gamma : \mathcal{O}[G]]_{\mathcal{O}} \cdot 
[ \mathcal{O}[G] : \leftindex^{\Gamma}{\mathcal{O}[G]}]_{\mathcal{O}},
\]
the second equality of \eqref{eq:index-max-eq} follows.

For the last statement, note that the hypotheses ensure that $K[G]$ is separable and hence maximal orders exist (see \S \ref{subsec:max-orders}). 
For any $\mathcal{O}$-order $\Lambda$ in $K[G]$, let $\mathrm{Disc}(\Lambda)$
denote the discriminant of $\Lambda$ with respect to the reduced trace map 
$\mathrm{tr} : K[G] \rightarrow K$.
Then $\mathrm{Disc}(\mathcal{M})$ is independent of the choice of maximal $\OO$-order $\mathcal{M}$ of $K[G]$ by \cite[(25.3)]{MR1972204}. Moreover, by \cite[(26.3)(iii)]{MR632548} we have $\mathrm{Disc}(\mathcal{O}[G]) = [\mathcal{M} : \mathcal{O}[G]]_{\mathcal{O}}^{2} \cdot \mathrm{Disc}(\mathcal{M})$, and so $[\mathcal{M} : \mathcal{O}[G]]_{\mathcal{O}}$
is independent of the choice of $\mathcal{M}$. 
\end{proof}

\subsection{The main theorem for group rings}
We now obtain a more precise version 
of Theorem \ref{thm:main} for lattices over group rings. 

\begin{theorem}\label{thm:gpringmain}
Let $\mathcal{O}$ be a Dedekind domain with field of fractions $K$. 
Assume that $K$ is a global field and that $\mathcal{O} \neq K$.
Let $G$ be a finite group such that $|G|$ is invertible in $K$.
Set $s=2$ if $G$ is abelian and $s=3$ otherwise.
Let $\mathcal{M}$ be a maximal $\mathcal{O}$-order in $K[G]$ containing $\mathcal{O}[G]$.
Let $\mathcal{K}$ be any nonzero ideal of $\mathcal{O}$.
Then there exists a nonzero ideal $\mathcal{I}$ of $\mathcal{O}$,
that can be chosen to be coprime to $\mathcal{K}$, with the following property:
given any $\mathcal{O}[G]$-lattice $X$ such that $KX$ is free of rank $n$ over $K[G]$,
there exists a free $\mathcal{O}[G]$-sublattice 
$Z$ of $X$ such that $[X : Z]_{\mathcal{O}}$ divides 
$\mathcal{I} \cdot [\mathcal{M} : \mathcal{O}[G]]_{\mathcal{O}}^{sn}$. 
Moreover, if $\mathcal{M}$ satisfies the equivalent conditions of Lemma \ref{lem:equiv-locally-free-implies-free}, then we can take $\mathcal{I}=\mathcal{O}$.
\end{theorem}

\begin{remark}
In the case $\mathcal{O}=\Z$, 
explicit conditions on $G$ under which $\mathcal{M}$ satisfies the equivalent conditions
of Lemma \ref{lem:equiv-locally-free-implies-free} are given in 
Proposition \ref{prop:trivialclassgroup} and Corollary \ref{cor:trivialclassgroup-abelian-case}.
\end{remark}

\begin{proof}[Proof of Theorem \ref{thm:gpringmain}]
We apply Theorem \ref{thm:main} with $\Lambda=\mathcal{O}[G]$.
If $G$ is abelian, then the desired result follows directly. 
Otherwise, by Proposition \ref{prop:index-in-two-orders}
we can and do take 
$J={\mathcal{O}[G]}^{\mathcal{M}}=\leftindex^{\mathcal{M}}{\mathcal{O}[G]}$, and we have
$[\mathcal{M} : J]_{\mathcal{O}}^{n} \cdot [\mathcal{M} : \Lambda]_{\mathcal{O}}^{n} =
[\mathcal{M} : \Lambda]_{\mathcal{O}}^{3n}$.
\end{proof}

\subsection{Jacobinski's formula and the index of a group ring in a maximal order}\label{subsec:jacobinskis-formula}
For further details on the following setup and notation, we refer the reader to 
\cite[\S 27]{MR632548} and the references therein.

Let $\mathcal{O}$ be a Dedekind domain with field of fractions $K\neq \OO$. 
Let $G$ be a finite group such that $|G|$ is invertible in $K$.
Then $K[G]$ is a separable finite-dimensional $K$-algebra.
We may write $K[G] = A_{1} \times \cdots \times A_{t}$,
where each $A_{i}$ is a simple $K$-algebra. 
For each $i$, let $K_{i}$ denote the centre of $A_{i}$.
Then each $K_{i}$ is a finite separable field extension of $K$, and there exist integers 
$n_{1}, \ldots, n_{t}$ such that $\dim_{K_{i}}A_{i} =n_{i}^{2}$ for each $i$.
Let $\mathrm{tr}_{i}$ denote the reduced trace from $A_{i}$ to $K$ (see \cite[\S 7D]{MR632548}).
Then 
$\tr_{i} = T_{K_{i}/K} \circ \tr_{A_{i}/K_{i}}$,
where $T_{K_{i}/K}$ is the ordinary trace from $K_{i}$ to $K$, and $\tr_{A_{i}/K_{i}}$
is the reduced trace from $A_{i}$ to $K_{i}$.

Let $\mathcal{M}$ be a maximal $\mathcal{O}$-order such that
$\mathcal{O}[G] \subseteq \mathcal{M} \subseteq K[G]$.
For each $i$, let $\mathcal{M}_{i}=\mathcal{M} \cap A_{i}$,
let $\mathcal{O}_{i}$ denote the integral closure of $\mathcal{O}$ in $K_{i}$,
and define the \textit{inverse different} of $\mathcal{M}_{i}$ with respect to $\mathrm{tr}_{i}$ 
to be 
$\mathcal{D}_{i}^{-1} = \{ x \in A_{i} : \mathrm{tr}_{i}(x\mathcal{M}_{i}) \subseteq \mathcal{O} \}$.
Then $\mathcal{M}=\mathcal{M}_{1} \times \cdots \times \mathcal{M}_{t}$ and each
$\mathcal{D}_{i}^{-1}$ is a two-sided $\mathcal{M}_{i}$-lattice containing $\mathcal{M}_{i}$.

\begin{theorem}[Jacobinski \cite{MR204538}]\label{thm:Jacobinski}
In the notation above, we have
\[ 
\leftindex^{\mathcal{M}}{\mathcal{O}[G]}
= {\mathcal{O}[G]}^{\mathcal{M}} 
= \bigoplus_{i=1}^{t} |G|n_{i}^{-1}\mathcal{D}_{i}^{-1}.
\]
\end{theorem}

A less explicit version of the following result is given in \cite[Proposition 3.6]{MR4356849}.

\begin{corollary}\label{cor:index-formula}
In the notation above, we have
\[
[\mathcal{M} : {\mathcal{O}[G]} ]_{\mathcal{O}}
= [\mathcal{O}[G] : \leftindex^{\mathcal{M}}{\mathcal{O}[G]} ]_{\mathcal{O}}
= \left( 
|G|^{|G|}
\prod_{i=1}^{t} \left(n_{i}^{[K_{i}:K]n_{i}^{2}}[\mathcal{D}_{i}^{-1}: \mathcal{M}_{i} ]_{\mathcal{O}}\right)^{-1}
\right)^{\frac{1}{2}},
\]
and this index is independent of the choice of $\mathcal{M}$.
\end{corollary}

\begin{proof}
By Theorem \ref{thm:Jacobinski} we have
\begin{align*}
[\mathcal{M} : \leftindex^{\mathcal{M}}{\mathcal{O}[G]} ]_{\mathcal{O}} 
&=\prod_{i=1}^{t} [\mathcal{M}_{i} : (|G|n_{i}^{-1}\mathcal{D}_{i}^{-1}) ]_{\mathcal{O}}\\
&=\prod_{i=1}^{t} (|G|n_{i}^{-1})^{\dim_{K}A_{i}}[\mathcal{M}_{i} : \mathcal{D}_{i}^{-1}]_{\mathcal{O}}\\
&= |G|^{|G|} \prod_{i=1}^{t} (n_{i}^{[K_{i}:K]n_{i}^{2}}[\mathcal{D}_{i}^{-1}: \mathcal{M}_{i} ]_{\mathcal{O}})^{-1},
\end{align*}
where in the last equality we have used that  $\dim_{K}A_{i} = [K_{i}:K]n_{i}^{2}$
for each $i$ and that $\prod_{i=1}^{t} \dim_{K} A_{i} = |G|$.
The desired result now follows from Proposition \ref{prop:index-in-two-orders}.
\end{proof}

\begin{corollary}\label{cor:rational-index-formula}
In the notation above, if $A_{i} \cong \mathrm{Mat}_{n_{i}}(K)$ for $i=1,\ldots,t$, then
\[
[\mathcal{M} : {\mathcal{O}[G]} ]_{\mathcal{O}}
= [\mathcal{O}[G] : \leftindex^{\mathcal{M}}{\mathcal{O}[G]} ]_{\mathcal{O}}
= \left( 
|G|^{|G|}
\prod_{i=1}^{t} n_{i}^{-n_{i}^{2}}
\right)^{\frac{1}{2}}.
\]
\end{corollary}

\begin{proof}
The hypotheses imply that $K_{i}=K$ and  
$\mathcal{D}_{i}^{-1} = \mathcal{M}_{i}$ for $i=1,\ldots,t$.
Thus the desired result follows from Corollary \ref{cor:index-formula}.
\end{proof}

A result similar to the following is given in \cite[p.\ 173, (11)]{MR4356849}.

\begin{corollary}\label{cor:abeliancasegeneral} 
In the notation above, if $G$ is abelian, then 
\[
[\mathcal{M} : {\mathcal{O}[G]} ]_{\mathcal{O}}
= [\mathcal{O}[G] : \leftindex^{\mathcal{M}}{\mathcal{O}[G]} ]_{\mathcal{O}}
= \left( 
|G|^{|G|}
\prod_{i=1}^{t} (\Delta_{K_{i}/K})^{-1}
\right)^{\frac{1}{2}},
\] 
where $\Delta_{K_{i}/K}$ denotes the discriminant of $\mathcal{O}_{i}$
with respect to $\mathcal{O}$.
\end{corollary}

\begin{proof}
Since $A$ is commutative, for every $i$ we have $n_{i}=1$, $A_{i}=K_{i}$, and $\mathcal{M}_{i}=\mathcal{O}_{i}$.
Thus the reduced trace $\tr_{i}$ coincides with the ordinary trace $T_{K_{i}/K}$
and so 
\[
\mathcal{D}_{i}^{-1} = \{ x \in K_{i} : T_{K_{i}/K}(x\mathcal{O}_{i}) \subseteq \mathcal{O}\}
\]
is the usual inverse different of $\mathcal{O}_{i}$ with respect to $\mathcal{O}$.
Moreover,
\[
[\mathcal{D}_{i}^{-1}: \mathcal{M}_{i} ]_{\mathcal{O}} 
= [\mathcal{D}_{i}^{-1}: \mathcal{O}_{i} ]_{\mathcal{O}}
= [\mathcal{O}_{i} : \mathcal{D}_{i}]_{\mathcal{O}}
= \Norm_{K_{i}/K}(\mathcal{D}_{i})
= \Delta_{K_{i}/K},
\]
where for the third equality, it suffices to first localise and then consider the determinant of the $K$-linear
endomorphism of $K_{i}$ given by multiplication by a generator of $\mathcal{D}_{i}$.
Therefore the desired result now follows from Corollary \ref{cor:index-formula}.
\end{proof}

We now make the last result completely explicit in the case $K=\Q$.

\begin{prop}\label{prop:abelianformula} 
Let $G$ be a finite abelian group and let $e$ denote its exponent.
Let $\mathcal{M}$ be the unique maximal $\Z$-order in $\Q[G]$. 
Then
\begin{equation*}\label{eq:abelianformula} 
{[\MM : \Z[G]] = \left( |G|^{|G|} \prod_{d \mid e }\biggl( d^{-\phi(d)} \prod_{p \mid d } p^{\frac{\phi(d)}{p-1}}\biggr)^{\hspace{-.1cm} t_d} \right)^{\hspace{-.1cm} \frac{1}{2}}},
\end{equation*}
where $t_{d}$ denotes the number of cyclic subgroups of $G$ of order $d$ and
$\phi(-)$ denotes the Euler totient function.
\end{prop}

\begin{proof} 
By \cite[Theorem 2]{MR252526}, we have
$\Q[G] \cong \prod_{d \mid e} \Q(\zeta_{d})^{(t_{d})}$,
where $\Q(\zeta_{d})^{(t_{d})}$ denotes the direct product of $t_{d}$ copies of $\Q(\zeta_{d})$
(see also \cite{MR34758}).
Moreover,
\[
\Delta_{\Q(\zeta_{d})/\Q}^{-1} = \biggr( d^{-\phi(d)}  \prod_{p \mid d } p^{\frac{\phi(d)}{p-1}} \biggr) \Z
\]
by \cite[Proposition 2.7]{MR1421575}.
Therefore the desired result now follows from a straightforward calculation
and Corollary \ref{cor:abeliancasegeneral} in the case $K=\Q$.
\end{proof}

The following special case is equivalent to \cite[Lemma 5.2]{MR2040875}, which was proven using 
different methods.

\begin{corollary}
Let $p$ be a prime, let $k$ be a positive integer, and let $G$ be the cyclic group of order $p^{k}$.
Let $\mathcal{M}$ be the unique maximal $\Z$-order in $\Q[G]$. 
Then 
\[
[\mathcal{M} : \Z[G]] = p^{1 + p + \cdots + p^{k-1}}.
\]
\end{corollary}

\subsection{The case of group rings over the integers}\label{subsec:integral-group-rings}
Let $\mathbb{H}$ denote the quaternion division algebra over $\R$.
For a finite group $G$, let $\Irr_{\C}(G)$ denote the set of complex irreducible characters of $G$.
Recall that $\chi \in \Irr_{\C}(G)$ is said to be an \textit{irreducible symplectic character}
if it is real-valued and the corresponding factor of $\R[G]$ is isomorphic to
the ring of $k \times k$ matrices over $\mathbb{H}$, for some positive integer $k$. 
For each $\chi \in \Irr_{\C}(G)$, let $\Q(\chi)$ denote the field generated by the values of $\chi$, and
let $C(\chi)$ be the narrow class group of $\Q(\chi)$ if $\chi$ is symplectic,
and the usual ideal class group of $\Q(\chi)$ otherwise. 

The following result is well-known to experts, but the authors were unable to locate it in this precise form in the literature.

\begin{prop}\label{prop:trivialclassgroup}
Let $G$ be a finite group and let $\mathcal{M}$ be a maximal $\Z$-order in $\Q[G]$.
Suppose that no factor of $\mathbb{R}[G]$ is isomorphic to the quaternions $\mathbb{H}$.
Then $\mathcal{M}$ satisfies the equivalent conditions of Lemma
\ref{lem:equiv-locally-free-implies-free} if and only if $C(\chi)$ is trivial for each $\chi \in \Irr_{\C}(G)$.
\end{prop}

\begin{proof}
The hypothesis on $\R[G]$ ensures that $\Q[G]$ satisfies the Eichler condition relative to 
$\Z$ (see \cite[\S 51A]{MR892316}). 
Hence Jacobinski's cancellation theorem \cite[(51.24)]{MR892316} implies that every $\Z$-order in $\Q[G]$, in particular $\mathcal{M}$, has the locally free cancellation property. 
Now write $\Q[G] = A_{1} \times \cdots \times A_{t}$, 
where each $A_{i}$ is a simple $\Q$-algebra.
For each $i$, let $K_{i}$ denote the centre of $A_{i}$ and let $\mathcal{M}_{i}=A_{i} \cap \mathcal{M}$. 
Then $\mathcal{M}=\mathcal{M}_{1} \times \cdots \times \mathcal{M}_{t}$ and
$\Cl(\mathcal{M}) \cong \Cl(\mathcal{M}_{1}) \oplus \cdots \oplus \Cl(\mathcal{M}_{t})$.
Each $K_{i}$ is isomorphic to $\Q(\chi)$ for some $\chi \in \Irr_{\C}(G)$
and by \cite[(49.32)]{MR892316} $\Cl(\mathcal{M}_{i})$ is isomorphic to $C(\chi)$. 
Therefore we obtain the `if' direction of the desired equivalence.
The `only if' direction now follows from the fact that 
$\{K_{i} : 1 \leq i \leq t \} = \{\Q(\chi) : \chi \in \Irr_{\C}(G) \}$.
\end{proof}

\begin{remark}\label{rmk:group-alg-max-order-canc-prop}
The hypothesis in Proposition \ref{prop:trivialclassgroup} that no factor of $\R[G]$ is isomorphic
to the quaternions $\mathbb{H}$ can be weakened to the requirement that $\mathcal{M}$
has the locally free cancellation property (or the stably free cancellation property).
Maximal $\Z$-orders with the locally free cancellation property in $\Q[G]$ for $G$ a 
binary polyhedral group have been classified in \cite[Theorem II]{MR703486}.
For example, for $2 \leq n \leq 5$ every maximal $\Z$-order in $\Q[Q_{4n}]$ satisfies the 
equivalent conditions of Lemma \ref{lem:equiv-locally-free-implies-free}, where
$Q_{4n}$ denotes the generalised quaternion group of order $4n$. 
See also Remark \ref{rmk:cancellation-max-order}.
\end{remark}

\begin{corollary}\label{cor:trivialclassgroup-abelian-case}
Let $G$ be a finite abelian group and let $e$ denote its exponent.
Define
\begin{multline*}
\Sigma = \{ 1, 2, 3, 4, 5, 6, 7, 8, 9, 10, 11, 12, 13, 14, 15, 16, 17, 18, 19, 20, 21, 22, 24, 25, \\
26, 27, 28, 30, 32, 33, 34, 35, 36, 38, 40, 42, 44, 45, 48, 50, 54, 60, 66, 70, 84, 90 \}
\end{multline*}
and let $\mathcal{M}$ be the unique maximal $\Z$-order in $\Q[G]$.
Then $\mathcal{M}$ satisfies the equivalent
conditions of Lemma \ref{lem:equiv-locally-free-implies-free} if and only if $e \in \Sigma$.
\end{corollary}

\begin{proof}
First write $G \cong C_{n_1} \times C_{n_2} \times \cdots \times C_{n_k}$ for positive integers 
$k, n_{1}, \ldots, n_{k}$ such that $n_{i} \mid n_{i+1}$ for $1 \leq i \leq k-1$.
Then $e=n_{k}$ and
$\{\Q(\chi) : \chi \in \Irr_{\C}(G) \} = \{ \Q(\zeta_{d}) : d \mid e \}$.
Since $\Z[\zeta_{d}]$ is the ring of integers of $\Q(\zeta_{d})$, this
implies that
\[
\{ C(\chi) :  \chi \in \Irr_{\C}(G) \} =  \{ \Cl(\Z[\zeta_{d}]) : d \mid e \}.
\]
The set $\Sigma$ is precisely the set positive integers $n$ for which $\Cl(\Z[\zeta_{n}])=0$ 
by \cite{MR429824}; see also \cite[Theorem 11.1]{MR1421575} 
(note that we include choices of $n$ such that $n \equiv 2 \bmod{4}$).
It can easily be checked that $\Sigma$ is also precisely the set of choices of $e$ 
for which $\Cl(\Z[\zeta_{d}])=0$ for all $d \mid e$.
Since no factor of $\Q[G]$ is isomorphic to the quaternions $\mathbb{H}$, the desired result
now follows from Proposition \ref{prop:trivialclassgroup}.
\end{proof}
 
\begin{remark}\label{rmk:trivial-locally-free-class-groups-of-group-rings}
The hypotheses of Proposition \ref{prop:trivialclassgroup} and Corollary \ref{cor:trivialclassgroup-abelian-case} are much weaker than $\Cl(\Z[G])$ itself being trivial. 
If $G$ is a finite abelian group,
then a result of Ph.\ Cassou-Nogu\`es \cite{MR0340393} shows that 
$\text{Cl}(\Z[G])$ is trivial if and only if
either $G \cong C_{2} \times C_{2}$ or $G \cong C_{n}$ where $1 \leq n \leq 11$ or $n \in \{13,14,17,19 \}$.
If $G$ is a finite non-abelian non-dihedral group,
then a result of End\^{o} and Hironaka \cite{MR519042} shows that 
$\text{Cl}(\Z[G])$ is trivial if and only if $G \cong A_{4}$, $S_{4}$ or $A_{5}$
(the if direction was already shown by Reiner and Ullom \cite{MR0367043}). 
For partial results in the dihedral case, see \cite{MR570730}.
\end{remark}

Many specialisations of Theorem \ref{thm:gpringmain} can 
now be obtained by applying the results of \S \ref{subsec:jacobinskis-formula}, 
Proposition \ref{prop:trivialclassgroup} and/or Corollary \ref{cor:trivialclassgroup-abelian-case};
the following is just one such example.

\begin{theorem}\label{thm:rat-rep-groups}
Let $G$ be a finite group and suppose that there exist positive integers $t, n_{1}, \ldots, n_{t}$ 
such that  $\Q[G] \cong \prod_{i=1}^{t} \mathrm{Mat}_{n_{i}}(\Q)$.
If $X$ is an $\Z[G]$-lattice such that $\Q X$ is free of rank $n$ over $\Q[G]$,
then there exists a free $\Z[G]$-sublattice $Z$ of $X$ such that $[X : Z]$ divides 
\[
\left( 
|G|^{|G|}
{\prod_{i=1}^{t}} n_{i}^{-n_{i}^{2}}
\right)^{\frac{3n}{2}}.
\]
\end{theorem}

\begin{proof}
Proposition \ref{prop:trivialclassgroup} implies that any maximal $\Z$-order in $\Q[G]$ satisfies the equivalent conditions of Lemma \ref{lem:equiv-locally-free-implies-free}.
Thus the result follows from Theorem \ref{thm:gpringmain} and Corollary \ref{cor:rational-index-formula}.
\end{proof}

\begin{remark}\label{rmk:collection-rat-rep-groups}
The collection of finite groups $G$ satisfying the hypothesis of Theorem \ref{thm:rat-rep-groups}
is closed under direct products and includes
symmetric groups and hyperoctahedral groups (e.g.\ the dihedral group of order $8$). See \cite{MR0357573} and \cite{MR765700} for more on this topic. 
\end{remark}

\section{Group rings modulo trace}\label{sec:group-rings-modulo-trace}

Let $\mathcal{O}$ be a Dedekind domain with field of fractions $K \neq \OO$. 
Let $G$ be a finite group such that $|G|$ is invertible in $K$ and let 
$\mathrm{Tr}_{G} = \sum_{g \in G} g$. 
Then both $K[G]$ and its quotient $K[G]/(\mathrm{Tr}_{G})$
are separable finite-dimensional $K$-algebras.
The purpose of this section is to consider lattices over the $\mathcal{O}$-order 
$\mathcal{O}[G]/(\mathrm{Tr}_{G})$.

Let $e=1-|G|^{-1}\mathrm{Tr}_{G}$, which is a central idempotent of $K[G]$. 
Let  $\pi_{e} : K[G] \rightarrow eK[G]$ be the projection map associated to $e$. 
Given a subset $X\subseteq K[G]$, let $X_{e}=\pi_{e}(X)$
and let $X^{{1-e}} = X\cap \ker(\pi_{e})$.
In particular, $K[G]_{e} = eK[G] \cong K[G]/(\mathrm{Tr}_{G})$ and $K[G]^{{1-e}}=\mathrm{Tr}_{G}K$.

Let $\Lambda=\OO[G]$ and let $\MM$ be a maximal 
$\OO$-order in $K[G]$ containing $\Lambda$. 
Then $\MM_e = e\MM$ is a maximal $\OO$-order of $K[G]_e$ containing $\Lambda_e$. 
By Proposition \ref{prop:index-in-two-orders},
$\lMM{\Lambda}$ is a two-sided ideal of $\MM$ contained in $\Lambda$.
Hence $(\leftindex^{{\MM}}{\Lambda})_{e}$ is a choice of two-sided ideal of $\MM_e$ contained in
$\Lambda_e$. This is not necessarily the largest such choice, but its form allows us to make 
use of our previous computations.

\begin{lemma}\label{l:augmentationindexcomp}
We have $[\MM_e:\Lambda_e]_\OO=|G|^{-1}[\MM:\Lambda]_\OO$ and $[\Lambda_e:(\leftindex^{{\MM}}{\Lambda})_{e}]_\OO = [\MM:\Lambda]_\OO$.
\end{lemma}

\begin{proof}
Consider the following diagram of $\OO$-lattices with exact rows
\[
\begin{tikzcd}
0 \ar{r}&\ar{r}  \Lambda^{1-e}  \ar[hookrightarrow]{d}  &\Lambda\ar{r} \ar[hookrightarrow]{d} &\ar{r}\Lambda_e  \ar[hookrightarrow]{d} &0\\
0 \ar{r} & \MM^{1-e} \ar{r} & \MM\ar{r}& \ar{r} \MM_e &0.
\end{tikzcd}
\]
Then by Lemma \ref{lem:index-diagram} we have
$[\MM:\Lambda ]_\OO 
=[\MM^{1-e}: \Lambda^{1-e}]_\OO \cdot 
[\MM_e:\Lambda_e]_\OO$. 
Note that
$\MM^{1-e}=\left(|G|^{-1}\mathrm{Tr}_{G} \right)\cdot \OO$ and 
$\Lambda^{1-e}=\MM^{1-e}\cap \Lambda
=\mathrm{Tr}_{G} \cdot \OO$. 
Hence $[\MM^{1-e}:\Lambda^{1-e}]_\OO= |G|$, and so we obtain the first equality.

Similarly, we also have the following diagram of $\OO$-lattices with exact rows
\begin{equation}\label{eq:1-e-diag}
\begin{tikzcd}
0 \ar{r}&\ar{r}  (\leftindex^{\MM}{\Lambda})^{1-e}  \ar[hookrightarrow]{d}  &\lMM{\Lambda}\ar{r} \ar[hookrightarrow]{d} &\ar{r}(\lMM{\Lambda})_e  \ar[hookrightarrow]{d} &0\\
0 \ar{r} & \Lambda^{1-e} \ar{r} & \Lambda\ar{r}& \ar{r} \Lambda_e &0.
\end{tikzcd} 
\end{equation}
Then by Lemma \ref{lem:index-diagram} we have 
$[\Lambda : \lMM{\Lambda}]_{\OO} = [\Lambda^{1-e} :  (\leftindex^{\MM}{\Lambda})^{1-e}]_{\OO} 
\cdot [\Lambda_{e} : (\lMM{\Lambda})_e ]_{\OO}$.
By maximality of $\lMM{\Lambda}$, the subset 
$(\leftindex^{\MM}{\Lambda})^{1-e}$ is the largest $\MM^{1-e}$-sublattice contained in $\Lambda^{1-e}$. 
Since $\MM^{1-e}\cong\OO$, we find that $\Lambda^{1-e}$, an $\OO$-lattice, is already a $\MM^{1-e}$-sublattice so that the left vertical map 
of \eqref{eq:1-e-diag} is an equality.
Hence we have 
$[\Lambda_e:(\leftindex^{{\MM}}{\Lambda})_{e}]_\OO= [\Lambda : \lMM{\Lambda}]_\OO$.
But $[\Lambda : \lMM{\Lambda}]_{\OO} = [\mathcal{M} : \Lambda]_{\OO}$ by 
Proposition \ref{prop:index-in-two-orders}, and thus we obtain the desired result.
\end{proof}

\begin{theorem}\label{thm:augmain}
Let $\mathcal{O}$ be a Dedekind domain with field of fractions $K$. 
Assume that $K$ is a global field and that $\mathcal{O} \neq K$.
Let $G$ be a finite group such that $|G|$ is invertible in $K$.
Set $s=2$ if $G$ is abelian and $s=3$ otherwise.
Let $\mathcal{M}$ be a maximal $\mathcal{O}$-order in $K[G]$ containing $\mathcal{O}[G]$.
Let $\mathcal{K}$ be any nonzero ideal of $\mathcal{O}$.
Then there exists a nonzero ideal $\mathcal{I}$ of $\mathcal{O}$,
that can be chosen to be coprime to $\mathcal{K}$, with the following property:
given any $\mathcal{O}[G]/(\Tr_G)$-lattice 
$X$ such that $KX$ is free of rank $n$ over $K[G]/(\Tr_{G})$,
there exists a free $\mathcal{O}[G]/(\Tr_G)$-sublattice 
$Z$ of $X$ such that $[X : Z]_{\mathcal{O}}$ divides 
$\mathcal{I} \cdot |G|^{-2n} \cdot [\mathcal{M} : \mathcal{O}[G]]_{\mathcal{O}}^{sn}$.
Moreover, if $\mathcal{M}$ satisfies the equivalent conditions of Lemma \ref{lem:equiv-locally-free-implies-free}, then we can take $\mathcal{I}=\mathcal{O}$.
\end{theorem}

\begin{proof}
Let $\Lambda=\mathcal{O}[G]$. 
Then $\MM_e$ is a maximal $\OO$-order of $K[G]_e$ containing $\Lambda_e = \OO[G]/(\Tr_G)$. 
Note that if $\mathcal{M}$ satisfies the equivalent conditions of 
Lemma \ref{lem:equiv-locally-free-implies-free}, then so does $\MM_e$.
By Lemma \ref{l:augmentationindexcomp} we have 
$[\mathcal{M}_{e} : \Lambda_{e}]_{\mathcal{O}} = |G|^{-1} [\mathcal{M} : \Lambda]_{\mathcal{O}}$.
By Proposition \ref{prop:index-in-two-orders}, $J:=(\leftindex^{\MM}{\Lambda})_e$
is a two-sided ideal of $\MM_e$ contained in $\Lambda_e$. 
Then we have
\[
[\mathcal{M}_{e} : \Lambda_{e}]_{\mathcal{O}} \cdot
[\mathcal{M}_{e} : J]_{\mathcal{O}} 
= [\mathcal{M}_{e} : \Lambda_{e}]_{\mathcal{O}}^{2} \cdot [\Lambda_{e} : J]_{\mathcal{O}}
= |G|^{-2} \cdot [\mathcal{M} : \Lambda]_{\mathcal{O}}^{3}.
\]
Therefore we obtain the desired result by applying Theorem \ref{thm:main} for the $\OO$-order $\Lambda_e$, 
the maximal $\OO$-order $\MM_e$ and the  ideal $J$.
\end{proof}

\section{Application: approximation of normal integral bases}\label{sec:application-approximation-of-normal-integral-bases}\label{sec:approx-NIB}

We refer the reader to \cite{MR717033} for an overview of normal integral bases, 
on which there is a vast literature. 
In this section, we consider examples of applications of the algebraic machinery of previous sections
to the approximation of normal integral bases.

Beyond the base field and the isomorphism type of the Galois group, the following result does not use any arithmetic information about the Galois extensions concerned.

\begin{theorem}\label{thm:approx-of-NIBS-wild}
Let $K$ be a number field and let $\mathcal{K}$ be any nonzero ideal of $\mathcal{O}_{K}$.
Let $G$ be a finite group and let $\mathcal{M}$ be a maximal $\mathcal{O}$-order in $K[G]$ 
containing $\mathcal{O}_{K}[G]$.
Set $s=2$ if $G$ is abelian and $s=3$ otherwise.
Then there exists a nonzero ideal $\mathcal{I}$ of $\mathcal{O}_{K}$,
that can be chosen to be coprime to $\mathcal{K}$, with the following property:
given any Galois extension $L/K$ with $\Gal(L/K) \cong G$, 
there exists $\alpha \in \mathcal{O}_{L}$ such that 
$[\mathcal{O}_{L} : \mathcal{O}_{K}[\Gal(L/K)] \cdot \alpha ]_{\mathcal{O}}$ divides
$\mathcal{I} \cdot [\mathcal{M} : \mathcal{O}_{K}[G]]_{\mathcal{O}_{K}}^{s} $. Moreover, if $\mathcal{M}$ satisfies the equivalent conditions of Lemma \ref{lem:equiv-locally-free-implies-free}, then we can take $\mathcal{I}=\mathcal{O}_{K}$. 
\end{theorem}

\begin{proof}
The normal basis theorem says that for a finite Galois extension of fields $L/K$
we have $L \cong K[\Gal(L/K)]$ as $K[\Gal(L/K)]$-modules.
Therefore the desired result now follows easily from Theorem \ref{thm:gpringmain} with $n=1$
and $\mathcal{O}=\mathcal{O}_{K}$.
\end{proof}

\begin{remark}\label{rmk:crude-ZG-max-index-bound}
An explicit formula for $[\mathcal{M} : \mathcal{O}_{K}[G]]_{\mathcal{O}_{K}}$ is given in Corollary \ref{cor:index-formula}. 
In particular, a weak but general bound is 
that $[\mathcal{M} : \mathcal{O}_{K}[G]]_{\mathcal{O}_{K}}^{s}$ divides $|G|^{\lceil s|G|/2 \rceil}$.
\end{remark}

By making the further assumption that the extensions concerned are at most tamely ramified, we obtain the following result with a stronger conclusion. 

\begin{theorem}\label{thm:approx-of-NIBS-tame}
Let $K$ be a number field, let $\mathcal{K}$ be any nonzero ideal of $\mathcal{O}_{K}$,
and $G$ be a finite group.
Then there exists a nonzero ideal $\mathcal{I}$ of $\mathcal{O}_{K}$,
that can be chosen to be coprime to $\mathcal{K}$, with the following property:
given any at most tamely ramified Galois extension $L/K$ with $\Gal(L/K) \cong G$, 
there exists $\alpha \in \mathcal{O}_{L}$ such that 
$[\mathcal{O}_{L} : \mathcal{O}_{K}[\Gal(L/K)] \cdot \alpha ]_{\mathcal{O}_{K}}$ divides
$\mathcal{I}$. 
Moreover, if $\mathcal{O}_{K}[G]$ satisfies the equivalent conditions of Lemma \ref{lem:equiv-locally-free-implies-free}, then we can take $\mathcal{I}=\mathcal{O}_{K}$.  
\end{theorem}

\begin{proof}
For an at most tamely ramified Galois extension $L/K$ with $\Gal(L/K) \cong G$,
we have that $\mathcal{O}_{L}$ is a locally free $\mathcal{O}_{K}[G]$-lattice of rank $1$
by \cite[Chapter I, \S 3, Corollary 2]{MR717033}, for example.
Therefore the desired result now follows easily from Proposition \ref{prop:free-sublattice-of-lf} with $\mathcal{O}=\mathcal{O}_{K}$ and $\Gamma=\mathcal{O}_{K}[G]$.
\end{proof}

\begin{remark}
Improved bounds can be obtained in special cases. 
For example, if $G$ is a finite group with no irreducible symplectic characters 
(e.g.\ $G$ is abelian or of odd order), 
then every (at most) tamely ramified Galois extension $L/\Q$ with $\Gal(L/\Q) \cong G$ has a normal integral basis by a special case of Taylor's proof \cite{MR608528}
of a conjecture of Fr\"ohlich (see \cite[\S I]{MR717033} for an overview).
Moreover, if $G$ is a finite abelian group, then Leopoldt's theorem \cite{MR108479}
(see also \cite{MR1037435}) implies that for every Galois extension 
$L/\Q$ with $\Gal(L/\Q) \cong G$, there exists $\alpha \in \mathcal{O}_{L}$
such that $[\mathcal{O}_{L} : \Z[\Gal(L/\Q)] \cdot \alpha]$ divides $[\mathcal{M} : \Z[G]]$,
where $\mathcal{M}$ is the unique maximal $\Z$-order in $\Q[G]$. 
By contrast, Theorems \ref{thm:approx-of-NIBS-wild} and \ref{thm:approx-of-NIBS-tame}
are very general and their short proofs use little or no arithmetic information about the particular field extensions concerned.
\end{remark}

\section{Application: approximation of strong Minkowski units}\label{sec:application-approximation-of-strong-minkowski-units}

In this section, we consider examples of applications of the algebraic machinery of previous sections
to the approximation of strong Minkowski units.

\begin{definition}\label{def:Minkowski-unit}
Let $L/K$ be a Galois extension of number fields and let $G=\Gal(L/K)$.
Let $\mu_{L}$ denote the roots of unity of $L$.
An element $\varepsilon \in \mathcal{O}_{L}^{\times} / \mu_{L}$ is said to be
\begin{enumerate}
\item a \textit{Minkowski unit} of $L/K$ if $\Q \otimes_{\Z} (\mathcal{O}_{L}^{\times}/\mu_{L}) = \Q[G] \cdot \varepsilon$,
\item a \textit{strong Minkowski unit} of $L/K$ if $\mathcal{O}_{L}^{\times} / \mu_{L} = \Z[G] \cdot \varepsilon$.
\end{enumerate}
\end{definition}

The following result is well known.

\begin{lemma}\label{lem:Hebrand-special-case}
Let $L/K$ be a Galois extension of number fields and let $G=\Gal(L/K)$.
If $K$ is equal to either $\Q$ or an imaginary quadratic field then $L/K$ has a Minkowski unit.
Moreover, if either $L$ is totally real or $K$ is imaginary quadratic, then 
$\Q \otimes_{\Z} (\mathcal{O}_{L}^{\times}/\mu_{L}) \cong \Q[G]/(\Tr_{G})$ as  
$\Q[G]/(\Tr_{G})$-modules (and as $\Q[G]$-modules).
\end{lemma}

\begin{proof}
By definition, $L/K$ has a Minkowski unit if and only if $\Q \otimes_{\Z} \mathcal{O}_{L}^{\times}$ is cyclic as a $\Q[G]$-module. By a theorem of Herbrand (see \cite[Chapter I, \S 4.3]{MR782485}, for example)
there is an isomorphism $(\Q\otimes_{\Z} \mathcal{O}_{L}^{\times}) \oplus \Q \cong \Q[S_{\infty}]$
of $\Q[G]$-modules, where $S_{\infty}$ denotes the set of infinite places of $L$. 
So the existence of a Minkowski unit is implied by 
$\Q[S_{\infty}]$ being cyclic as a $\Q[G]$-module, which is equivalent to $S_{\infty}$ being transitive as a $G$-set. This occurs precisely when $K$ has a unique infinite place. 
If either $L$ is totally real or $K$ is imaginary quadratic, then the unique infinite place of $K$ splits completely in $L/K$ and thus $\Q[S_{\infty}] \cong \Q[G]$ as $\Q[G]$-modules. 
\end{proof}

\begin{remark}\label{rmk:exist-strong-Minkowski}
The existence of strong Minkowski units
(which some authors refer to as Minkowski units)
in special cases has been studied in numerous articles,
including \cite{MR0404203, MR0376610, MR526783, MR612684, MR526781, MR537446, MR729740, MR0374084, MR971088, MR1089790, MR1173123, MR1189093, MR2039416, MR2152220, MR2747023, MR2865419},
as well as \cite[\S 3.3 \& \S 3.5.1]{MR2078267} and the references therein.
Also see the articles cited below.
\end{remark}

\begin{remark}
If $L/\Q$ is a Galois extension with $[L:\Q]$ odd, then $L$ is totally real and
$\mathcal{O}_{L}^{\times} / \mu_{L} \cong 
\{ u \in \mathcal{O}_{L}^{\times} : \mathrm{Norm}_{L/\Q}(u) =1 \}$ as $\Z[\Gal(L/\Q)]$-lattices.
\end{remark}

The following result is a strong refinement and generalisation
of \cite[Theorem 1]{MR1418221}, \cite[Proposition 5.2]{MR3035966} and \cite[Theorem 3.3]{MR3145320}, which only consider finite cyclic or abelian totally real extensions of $\Q$ 
and do not actually bound the index in question.

\begin{theorem}\label{thm:main-Minkowski}
Let $G$ be a finite group and let $k$ be a positive integer. 
Set $s=2$ if $G$ is abelian and $s=3$ otherwise.
Then there exists a positive integer $i$, which can be chosen to be coprime to $k$, with the following property:
given any finite Galois extension $L/K$ with $\Gal(L/K)\cong G$ and $K$ equal to either $\Q$ or an imaginary quadratic field,
there exists a Minkowski unit $\varepsilon \in \mathcal{O}_{L}^{\times}/\mu_{L}$ 
such that $[\mathcal{O}_{L}^{\times}/\mu_{L} : \Z[\Gal(L/K)] \cdot \varepsilon]$ divides 
$i \cdot |G|^{-2} [\mathcal{M} : \Z[G]]^{s}$, 
where $\mathcal{M}$ is a maximal $\Z$-order in $\Q[G]$ containing $\Z[G]$.
\end{theorem}

\begin{remark}\label{rmk:weak-max-order-bound-Minkowski}
An explicit formula for $[\mathcal{M} : \Z[G]]$ is given in Corollary \ref{cor:index-formula}. 
In particular, a weak but general bound is 
that $|G|^{-2} [\mathcal{M} : \Z[G]]^{s}$ divides $|G|^{\lceil s|G|/2 \rceil -2}$.
\end{remark}

\begin{remark}
It is interesting to compare Theorem \ref{thm:main-Minkowski} 
with \cite[Theorem 1.1]{MR3610268},
which in the case that $L/\Q$ is a finite Galois extension asserts the existence of a Minkowski unit
$\varepsilon$ such that the index $[\mathcal{O}_{L}^{\times} / \mu_{L} : \Z[\Gal(L/\Q)] \cdot \varepsilon]$ is bounded by an expression involving the Weil height of $\varepsilon$, 
the regulator of $L$, the degree $[L:\Q]$ and $\rank_{\Z} \mathcal{O}_{L}^{\times}$.
\end{remark}

\begin{proof}[Proof of Theorem \ref{thm:main-Minkowski}]
Let $\Gamma = \Z[G] / (\Tr_{G})$. 
By Theorem \ref{thm:augmain} there exists a positive integer $i$,
which can be chosen to be coprime to $k$,
with the following property: given any $\Gamma$-lattice $X$ such that $\Q X \cong  \Q[G] / (\Tr_{G})$
as $\Q[G] / (\Tr_{G})$-modules, 
there exists a free $\Gamma$-sublattice $Y$ of $X$ such that $[X : Y]$ divides 
$i \cdot |G|^{-2} [\mathcal{M} : \Z[G]]^{s}$.
Note that if $\varepsilon \in X$ is a free $\Gamma$-generator of $Y$, 
then $Y=\Gamma \cdot \varepsilon = \Z[G] \cdot \varepsilon$.
For $L/K$ with either $L$ totally real or $K$ imaginary quadratic,
the desired result now follows from Lemma \ref{lem:Hebrand-special-case}
after fixing an isomorphism $\Gal(L/K) \cong G$. 
For $L/K$ with $L$ totally imaginary and $K=\Q$, the result follows from 
 Lemma \ref{lem:Hebrand-special-case} and Remark \ref{rmk:non-free}.
\end{proof}

\begin{corollary}\label{cor:Minkowski-good-M}
Let $G$ be a finite group and let $\mathcal{M}$ be a maximal $\Z$-order in $\Q[G]$ containing $\Z[G]$.
Suppose that $\mathcal{M}$ satisfies the equivalent conditions of Lemma \ref{lem:equiv-locally-free-implies-free} (see \S \ref{subsec:integral-group-rings} for conditions on $G$ under which this holds).
Set $s=2$ if $G$ is abelian and $s=3$ otherwise.
Then given any finite Galois extension $L/K$ with $\Gal(L/K)\cong G$ and $K$ equal to either $\Q$ or an imaginary quadratic field,
there exists a Minkowski unit $\varepsilon \in \mathcal{O}_{L}^{\times}/\mu_{L}$ 
such that $[\mathcal{O}_{L}^{\times}/\mu_{L} : \Z[\Gal(L/K)] \cdot \varepsilon]$ divides 
$|G|^{-2} [\mathcal{M} : \Z[G]]^{s}$.
\end{corollary}

\begin{proof}
Under the hypotheses on $\mathcal{M}$, the desired result follows as in the proof
of Theorem~\ref{thm:main-Minkowski} after noting that we can take $i=1$ in the application of 
Theorem \ref{thm:augmain}.
\end{proof}

\begin{remark}
Let $L/\Q$ be a finite Galois extension such that $L$ is CM and let $L^{+}$ denote its maximal
totally real subfield. Let $\varepsilon \in \mathcal{O}_{L^{+}}^{\times}/\{ \pm 1 \}$ be a Minkowski 
unit of $L^{+}/\Q$ and by abuse of notation let this also denote its image in 
$\mathcal{O}_{L}^{\times}/\mu_{L}$. 
By \cite[Theorem 4.12]{MR1421575}
we have that  $[\mathcal{O}_{L}^{\times} : \mu_{L} \mathcal{O}_{L^{+}}^{\times}]=1$ or $2$, and so
$[\mathcal{O}_{L}^{\times}/\mu_{L} : \Z[\Gal(L/\Q)] \cdot \varepsilon]$ divides 
$2[\mathcal{O}_{L^{+}}^{\times} / \{ \pm 1 \} : \Z[\Gal(L^{+}/\Q)] \cdot \varepsilon]$.
Thus $\varepsilon$ is also a Minkowski unit of $L/\Q$, 
and we obtain stronger analogues of Theorem \ref{thm:main-Minkowski} and Corollary \ref{cor:Minkowski-good-M} in this situation.
\end{remark}

The above results can be strengthened for extensions of prime degree.

\begin{theorem}\label{thm:deg-p-Minkowski}
Let $p$ be an odd prime and let $k$ be a positive integer.
Then there exists a positive integer $i$, which can be chosen 
to be coprime to $k$, with the following property: 
given any cyclic field extension $L/K$ with $[L:K]=p$
and $K$ equal to either $\Q$ or an imaginary quadratic field, 
there exists a Minkowski unit
$\varepsilon \in \mathcal{O}_{L}^{\times} / \mu_{L}$ such that
$[\mathcal{O}_{L}^{\times} / \mu_{L} : \Z[\Gal(L/K)] \cdot \varepsilon]$ divides $i$.
\end{theorem}

\begin{proof}
Let $G$ be the cyclic group of order $p$ and let $\mathcal{M} = \Z[G] / (\Tr_{G})$. 
Then $\mathcal{M} \cong \Z[\zeta_{p}]$, which is a maximal $\Z$-order.
By Corollary \ref{cor:free-sublattice-of-lf-max-case} there exists a positive integer $i$,
which can be chosen to be coprime to $k$,
with the following property: given any $\mathcal{M}$-lattice $X$ such that $\Q X$
is free of rank $1$ as a $\Q[G] / (\Tr_{G})$-module, 
there exists a free $\mathcal{M}$-sublattice $Y$ of $X$ such that $[X : Y]$ divides $i$.
Note that if $\varepsilon \in X$ is a free $\mathcal{M}$-generator of $Y$, 
then $Y=\mathcal{M} \cdot \varepsilon = \Z[G] \cdot \varepsilon$.
The desired result now follows from Lemma \ref{lem:Hebrand-special-case}
since the hypotheses ensure that either $L$ is totally real or $K$ is imaginary quadratic.
\end{proof}

The following result is not new, as it is the combination of \cite{MR0217045}
(see also \cite[Corollary]{MR0244193}) and an easy consequence of \cite[Th\'eor\`em]{MR589107};
we include it for completeness.

\begin{corollary}\label{cor:deg-p-Minkowski}
Let $p$ be a prime such that $3 \leq p \leq 19$. 
Then every cyclic field extension $L/K$ with $[L:K]=p$ and $K$ equal to either $\Q$ or an imaginary quadratic field has a strong Minkowski unit.
\end{corollary}

\begin{proof}
In the proof of Theorem \ref{thm:deg-p-Minkowski}, adding the hypothesis that $p \leq 19$
implies that $\mathcal{M} \cong \Z[\zeta_{p}]$ has trivial class group and so
satisfies the equivalent conditions of Lemma ~\ref{lem:equiv-locally-free-implies-free}.
Hence we can take $i=1$ by Remark \ref{rmk:I-trivial}, and this implies the desired result.
\end{proof}

\begin{remark}
It is interesting to compare Theorem \ref{thm:deg-p-Minkowski} to (i) \cite[Theorem]{MR0244193}
when $K=\Q$ and (ii) \cite[Th\'eor\`em]{MR589107} when $K$ is imaginary quadratic.
Result (i) considers cyclic extensions $L/\Q$ of odd prime degree $p$ and gives sufficient conditions
on ideals of $\Z[\zeta_{p}]$ of norm equal to the class number $h_{L}$ of $\mathcal{O}_{L}$ 
for both the existence and non-existence of a strong Minkowski unit of $L/\Q$. 
The proof uses the fact that $\mathcal{O}_{L}^{\times}/\{\pm 1\}$ contains a free $\Z[\zeta_{p}]$-submodule of index $h_{L}$ generated by a cyclotomic unit. 
Result (ii) is analogous and uses elliptic units. By contrast, the proof and statement of Theorem \ref{thm:deg-p-Minkowski} do not depend on the particular extension $L/K$.
\end{remark}

\section{Application: rational points on abelian varieties}\label{sec:application-rational-points-on-abelian-varieties}

In this section, we consider examples of applications of the algebraic machinery of previous sections
to the Galois module structure of rational points of abelian varieties. By the Mordell--Weil theorem, 
for every abelian variety $A$ over a number field $K$, the group $A(K)/A(K)_{\tors}$ is a free $\Z$-module of finite rank. If $L/K$ is a Galois extension of number fields, then $A(L)/A(L)_{\tors}$ is a 
$\Z[\Gal(L/K)]$-lattice, so is amenable to study via our methods.

\begin{theorem}\label{thm:main-elliptic}
Let $G$ be a finite group and let $k$ be a positive integer.
Set $s=2$ if $G$ is abelian and $s=3$ otherwise.
Then there exists a positive integer $i$, which can be chosen to be coprime to $k$,
with the following property:
given any Galois extension of number fields $L/K$ with $\Gal(L/K)\cong G$,
and any abelian variety $A/K$ such that $\Q \otimes_{\Z} A(L)$ is cyclic as a $\Q[G]$-module, 
there exists $\varepsilon \in A(L)/A(L)_{\tors}$ such that
$[A(L)/A(L)_{\tors} : \Z[G] \cdot \varepsilon]$ is finite 
and divides $i \cdot [\mathcal{M} : \Z[G] ]^{s}$, 
where $\mathcal{M}$ is any maximal $\Z$-order in $\Q[G]$ containing $\Z[G]$. 
\end{theorem}

\begin{remark}
An explicit formula for $[\mathcal{M} : \Z[G]]$ is given in Corollary \ref{cor:index-formula}. In particular, a weak but general bound is that $[\mathcal{M} : \Z[G]]^{s}$ divides $|G|^{\lceil s|G|/2 \rceil }$.
\end{remark}

\begin{remark}\label{r:artinsinductiontheorem}
Let $G$ be a finite group.
The isomorphism class of a finite-dimensional $\Q[G]$-module $V$ is entirely determined by the values of $\dim_{\Q} V^{H}$ as $H$ runs over a set of representatives of the set of cyclic subgroups of $G$ up to conjugacy (see \cite[\S 13.1, Corollary to Theorem 30\textprime]{MR0450380}).
In particular, $V$ is free of rank $1$ if and only if $\dim_{\Q} V^{H} = [G:H]$ for all cyclic subgroups
$H$ of $G$ up to conjugacy.
\end{remark}

\begin{proof}[Proof of Theorem \ref{thm:main-elliptic}]
By Theorem \ref{thm:gpringmain}, there exists a positive integer $i$,
which can be chosen to be coprime to $n$,
with the following property: 
given any $\Z[G]$-lattice $X$ such that $\Q \otimes_{\Z} X$ is free of rank $1$
as a $\Q[G]$-module, 
there exists a free $\Z[G]$-sublattice $Y$ of $X$ such that $[X : Y]$ divides 
$i \cdot [\mathcal{M} : \Z[G]]^{s}$. 
By Remark \ref{rmk:non-free}, $i$ also has the property that 
given any $\Z[G]$-lattice $X$ such that $\Q \otimes_{\Z} X$ is cyclic
as a $\Q[G]$-module, 
there exists a cyclic $\Z[G]$-sublattice $Y$ of $X$ such that $[X : Y]$ divides 
$i \cdot [\mathcal{M} : \Z[G]]^{s}$. 
In particular, this holds for $X=A(L)/A(L)_{\tors}$ after fixing an isomorphism $G \cong \Gal(L/K)$.
\end{proof}

\begin{theorem}\label{t:abelianvarcyclicp}
Let $p$ be an odd prime and let $k$ be a positive integer.
Then there exists a positive integer $i$, which can be chosen 
to be coprime to $k$, with the following property: 
given any cyclic extension $L/K$ of number fields with $[L:K]=p$
and any abelian variety $A/K$ such that $\rank_{\Z}A(K)=0$ and $\rank_{\Z}A(L)=p-1$, 
there exists $\varepsilon \in A(L)/A(L)_{\tors}$ such that
$[A(L)/A(L)_{\tors} : \Z[\Gal(L/K)] \cdot \varepsilon]$ is finite and divides $i$.
\end{theorem}

\begin{proof} 
Let $G$ be the cyclic group of order $p$ and let $\mathcal{M} = \Z[G] / (\Tr_{G})$. 
Then $\mathcal{M} \cong \Z[\zeta_{p}]$, which is a maximal $\Z$-order.
By Corollary \ref{cor:free-sublattice-of-lf-max-case} there exists a positive integer $i$,
which can be chosen to be coprime to $n$,
with the following property: 
given any $\mathcal{M}$-lattice $X$ such that $\Q X \cong  \Q[G] / (\Tr_{G})$
as $\Q[G] / (\Tr_{G})$-modules, 
there exists a free $\mathcal{M}$-sublattice $Y$ of $X$ such that $[X : Y]$ divides $i$.
After fixing an isomorphism $G \cong \Gal(L/K)$,
the desired result now follows since the rank hypotheses ensure that 
$\Q \otimes_{\Z} (A(L)/A(L)_{\tors}) \cong \Q[G]/(\Tr_{G})$ as $\Q[G]$-modules 
(see Remark \ref{r:artinsinductiontheorem}) and hence as $\Q[G]/(\Tr_{G})$-modules.
\end{proof}

\begin{remark}
Note that for $\Q \otimes_{\Z}(A(L)/A(L)_{\tors})$ to be cyclic as a ${\Q}[G]/(\Tr_{G})$-module, it is necessary that $\rank_{\Z} A(L)=0$ or $p-1$.
\end{remark}

\begin{corollary}
Let $p$ be a prime such that $3 \leq p \leq 19$. 
Then given any cyclic extension of number fields $L/K$ with $[L:K]=p$ and any abelian variety $A/K$
such that $\rank_{\Z}A(K)=0$ and $\rank_{\Z}A(L)=p-1$, 
there exists $\varepsilon \in A(L)/A(L)_{\tors}$ such that
$A(L)/A(L)_{\mathrm{tors}} = \Z[\Gal(L/K)] \cdot \varepsilon$. 
\end{corollary}

\begin{proof}
In the proof of Theorem \ref{t:abelianvarcyclicp}, the additional hypothesis that $p \leq 19$
implies that $\mathcal{M} \cong \Z[\zeta_{p}]$ has trivial class group and so
satisfies the equivalent conditions of Lemma ~\ref{lem:equiv-locally-free-implies-free}.
Hence we can take $i=1$ by Remark \ref{rmk:I-trivial}, and this implies the desired result.
\end{proof}

\begin{theorem}\label{thm:complex-mult}
Let $p$ be a prime, let $F$ be an imaginary quadratic field with discriminant coprime to $p$,
and let $\mathcal{K}$ be a nonzero ideal of $\mathcal{O}_{F}$. 
Then there exists an ideal $\mathcal{I}$ of $\mathcal{O}_{F}$, which can be chosen to be coprime
to $\mathcal{K}$, with the following property: 
given any cyclic extension of number fields $L/K$ with $[L:K]=p$ and such that $K$ contains $F$, 
and any elliptic curve $E/K$ with complex multiplication by $\mathcal{O}_{F}$ and with $\rank_{\Z} E(K)=0$
and $\rank_{\Z} E(L)=2(p-1)$, there exists 
$\varepsilon \in E(L)/E(L)_{\tors}$ such that $[E(L)/E(L)_{\tors} : \mathcal{O}_{F}[\Gal(L/K)] \cdot \varepsilon ]_{\mathcal{O}_{F}}$ divides $\mathcal{I}$.
\end{theorem}

\begin{proof}
Let $G$ be the cyclic group of order $p$. 
The discriminant of $\Z[\zeta_{p}]$ is a power of $p$ and in particular is coprime to the discriminant
of $\mathcal{O}_{F}$. 
Hence $\Q(\zeta_{p})$ and $F$ are linearly disjoint over $\Q$.
Moreover, by \cite[III, \S 3, Proposition 17]{MR1282723} we have 
$\mathcal{O}_{F}[\zeta_{p}] = \mathcal{O}_{F(\zeta_{p})}$.
Thus $\mathcal{O}_{F}[G]/(\Tr_{G})\cong \mathcal{O}_{F}[\zeta_{p}]$ is a maximal 
$\mathcal{O}_{F}$-order.
By Corollary \ref{cor:free-sublattice-of-lf-max-case} there exists a nonzero ideal $\mathcal{I}$ of 
$\mathcal{O}_{F}$, which can be chosen to be coprime to $\mathcal{K}$, with the following property: 
given any $\mathcal{O}_{F}[G]/(\Tr_{G})$-lattice $X$ such that $F \otimes_{\mathcal{O}_{F}} X \cong  F[G] / (\Tr_{G})$
as $F[G] / (\Tr_{G})$-modules, 
there exists a free $\OO_{F}[G]/(\Tr_{G})$-sublattice $Y$ of $X$ such that $[X : Y]_{\mathcal{O}_{F}}$ divides $\mathcal{I}$.

Let $E$, $L$ and $K$ be as in the theorem and fix an isomorphism $G \cong \Gal(L/K)$.
By assumption, $E$ has CM by $\OO_F$ defined over $K$. 
The commuting Galois action and action by endomorphisms then give $E(L)/E(L)_\tors$ the structure of an $\OO_F[G]$-lattice. Moreover, as $\rk_\z E(K)=0$, it is in fact a $\OO_F[G]/(\Tr_G)$-lattice and since $\rank_{\z} E(L)=2(p-1)$, we have that $\dim_{\q} F\ot _{\OO_{F}}(E(L)/E(L)_\tors)=2(p-1)$. Since $F$ and $\q(\zeta_{p})$ are linearly disjoint, the unique $F[G]/(\Tr_{G})$-module with these properties is $F[G]/(\Tr_{G})$ itself. Therefore $E(L)/E(L)_\tors$ is an example of an $\OO_F[G]/(\Tr_G)$-lattice
such that $F \otimes_{\mathcal{O}_{F}} X \cong  F[G] / (\Tr_{G})$.
\end{proof}

\begin{remark}
Since $\Q(\zeta_{p})$ and $F$ are linearly disjoint over $\Q$, 
the ${F}[G]/(\Tr_{G})$-module $F\otimes_{\mathcal{O}_{F}}(E(L)/E(L)_{\tors})$ is cyclic
if and only if either $\rank_{\Z} E(L)=0$ or $2(p-1)$.
\end{remark}

\begin{corollary}
Let $p$ be a prime, let $F$ be an imaginary quadratic field with discriminant coprime to $p$ such that $\mathcal{O}_{F(\zeta_{p})}$ has trivial class group.
Then for every cyclic extension of number fields $L/K$ such
that $K$ contains $F$ and $[L:K]=p$, and for every elliptic curve $E/K$ with complex multiplication by $\OO$ and with $\rank_{\Z} E(K)=0$ and $\rank_{\Z} E(L)= 2(p-1)$, we have that $E(L)/E(L)_{\tors}$ is free as an $\OO_{F}[G]/(\Tr_{G})$-module.
\end{corollary}

\begin{proof}
In the proof of Theorem \ref{thm:complex-mult}, the additional hypothesis that
$\mathcal{O}_{F(\zeta_{p})}$ has trivial class group ensures that 
$\mathcal{O}_{F}/(\Tr_{G}) \cong \mathcal{O}_{F(\zeta_{p})}$ satisfies the equivalent conditions of Lemma \ref{lem:equiv-locally-free-implies-free}. 
Hence we can take  $\mathcal{I}=\mathcal{O}_{F}$ by Remark \ref{rmk:I-trivial}, and this implies the desired result.
\end{proof}

\bibliography{Minkowski-approximation-bib}
\bibliographystyle{amsalpha}

\end{document}